\documentclass[11pt,fleqn]{amsart}
\usepackage{latexsym,amsmath,amssymb,amsfonts,epsfig}
\usepackage[initials]{amsrefs}
\usepackage[all]{xy}
\usepackage{eufrak}

\providecommand{\lra}{\longrightarrow}
\providecommand{\hra}{\hookrightarrow}

\providecommand{\CC}{{\mathbb{C}}}
\providecommand{\RR}{{\mathbb{R}}}

\providecommand{\OO}{{\mathcal O}}

\providecommand{\GG}{{\mathcal G}}

\providecommand{\THM}{{\mathbb{T}_H M}}

\providecommand{\pareq}{{\sim_H}}
\providecommand{\Lg}{{\mathfrak g}}

\providecommand{\sm}{{\Gamma}}%{{\Gamma^\infty}}  % `sm' for smooth sections 
\newtheorem{definition}{Definition}
\newtheorem{lemma}[definition]{Lemma}
\newtheorem{proposition}[definition]{Proposition}
\newtheorem{corollary}[definition]{Corollary}

\begin{document}
\title{The Geometry of the Osculating Nilpotent Group Structures of the  Heisenberg Calculus}
%\title{The Intrinsic Geometry of the Osculating Structures that Underlie the Heisenberg Calculus (or Why the Tangent Space in sub-Riemannian Geometry Is a Group)}
\author{Pierre Julg}
\address{Universit\'e d'Orl\'eans, MAPMO,  Route de Chartres, BP 6759, 45067 Orl\'eans Cedex 2, France}
\email{pierre.julg@univ-orleans.fr}
\author{Erik van Erp}
\address{Department of Mathematics, Dartmouth College, 27 N. Main St., Hanover, New Hampshire 03755, USA}
\email{jhamvanerp@gmail.com}\maketitle

{\small
\begin{center}{\bf Abstract}\end{center}
\vskip 6pt

\begin{quote}We explore the geometry that underlies the osculating nilpotent group structures of the Heisenberg calculus.
For a smooth manifold $M$ with a distribution $H\subseteq TM$ analysts use explicit (and rather complicated) coordinate formulas to define the nilpotent groups that are central to the calculus.
Our aim in this paper is to provide insight in the intrinsic geometry that underlies these coordinate formulas.
First, we introduce `parabolic arrows' as a  generalization of tangent vectors. The definition of parabolic arrows involves a mix of first and second order derivatives.
Parabolic arrows can be composed, and the group of parabolic arrows can be identified with the nilpotent groups of the (generalized) Heisenberg calculus.
Secondly, we formulate a notion of exponential map for the fiber bundle of parabolic arrows, and show how it explains the coordinate formulas of osculating structures found in the literature on the Heisenberg calculus.
The result is a conceptual simplification and unification of the treatment of the Heisenberg calculus.
\end{quote}
} 

\section{Introduction: Osculating nilpotent groups in analysis and geometry}

\subsection{Motivation}

The ideas in this paper were motivated by the work of the second author on  index problems for hypoelliptic Fredholm operators in the Heisenberg calculus \citelist{\cite{vE10a} \cite{vE10b}}.
The construction of an appropriate `tangent groupoid' for this index problem required a better geometric understanding of the nilpotent group structures that  appear in the definition of the Heisenberg pseudodifferential calculus.
We believe that the geometric ideas that we developed are of independent interest, and in the present paper we present them without reference to index theoretical concerns.
This introduction provides an exposition of the history of the problem.

\subsection{Nilpotent groups in analysis}

Osculating group structures were first introduced by Gerald Folland and Elias Stein \cite{FS74} as an aid in the analysis of the tangential CR operator $\bar{\partial}_b$ on the boundary of a strongly pseudoconvex complex domain. 
Let $H^{1,0}\subset TM\otimes \CC$ be a CR structure on $M$.
At each point $m\in M$, Folland and Stein consider a special type of coordinates $x\in \RR^{2k+1}$ with $x=0$ at $m$.
The coordinate system needs to be adapted, in a suitable sense, to the CR structure.
The coordinate space $\RR^{2k+1}$ is  identified with the Heisenberg group,
and on $\RR^{2k+1}$
there are  vector fields $X_1,\cdots, X_k, Y_1, \cdots, Y_k, T$ that are  translation invariant for the Heisenberg group structure, with commutator relations $[X_j, Y_j]=T$.
These vector fields on $\RR^{2k+1}$ are identified, via the carefully chosen coordinate system, with vector fields on $M$ defined in a neighborhood of $m$.

The choice of the coordinate system is such that the complex vector fields
\[ Z_j = \frac{1}{2}(X_j+\sqrt{-1}\,Y_j)\]
are `close to' an orthonormal frame for the bundle $H^{1,0}$.
Exactly what it means to be `close to' (or to {\em `osculate'}) is made precise in \cite[Theorem~14.1]{FS74}.
Their precise notion of  `closeness'  allows a reduction of the hypoellipticity problem for the $\bar{\partial}_b$ operator on a CR manifold to the hypoellipticity problem for translation invariant model operators on the Heisenberg group,
which, in turn, is solved by noncommutative harmonic analysis. 

Folland and Stein refer to their special coordinate systems as {\em osculating Heisenberg structures}.
Osculating structures of nilpotent groups  play a key role in the subsequent literature on hypoelliptic operators (for example \citelist{\cite{BG88} \cite{Cu89} \cite{Ta84} \cite{RS76}}).  
Following Folland and Stein, analysts typically choose local coordinates $U\to \RR^n$ on an open set $U\subseteq M$,
and define a nilpotent group structure on the coordinate space $\RR^n$
by means of an explicit formula.

For example, in \cite{BG88} Richard Beals and Peter Greiner introduce osculating structures for {\em Heisenberg manifolds}.
Heisenberg manifolds are  manifolds equipped with a distribution $H\subset TM$ of codimension one.
The distribution $H$ may be a contact structure, a foliation, or it can be an arbitrary distribution with no geometric significance.
Beals and Greiner start with a system of coordinates $\psi_m$ that depends on the point $m\in M$
and is loosely adapted to the distribution $H$.
Each point $m\in M$ has its own coordinate system $\psi_m\,\colon U_m\to \RR^n$ for which $\psi_m(m)=0$,
while the coordinate systems $\psi_m$ themselves have to vary smoothly with $m$ (in a neighborhood of $m$).
The group structure on the coordinate space $\RR^n$ is not fixed, but is defined by means of a rather complicated  string of formulas
(1.8), (1.11), (1.15) in \cite{BG88}.
These formulas contain partial derivatives that measure the {\em changes} in the coordinate system $\Phi_m$ as the point $m$ varies.
In this way, the nilpotent group structure on $\RR^n$ depends, in a rather complicated way, on the germ of the distribution $H$ near $m$. 

As a final example, in \cite{Cu89} Thomas Cummins generalizes the Heisenberg pseudodifferemtial calculus further to manifolds with a {\em filtration} $H_1\subseteq H_2\subseteq H_3=TM$, giving rise to three-step nilpotent osculating groups.
Cummins' formulas for the osculating group structure are very similar to (but more general than) the formulas of Beals and Greiner.

It is hard to discern any intrinsic geometry underlying the coordinate formulas that define the osculating group structures.
The aim of this paper is to clarify the geometric origin of these group structures,
and to derive the formulas found in the analytic literature from geometrically intuitive principles.

\subsection{Nilpotent Lie algebras and the equivalence problem}

In the geometric literature,  group structures that are closely related to the osculating structures of the analysts have been around for some time.
The oldest mention of such structures is in work on the `equivalence problem' introduced by Cartan in 1910 \cite{Ca10}: Find a full set of infinitesimal invariants of a manifold with distribution $H\subseteq M$.

There exists a simple and elegant definition of the {\em Lie algebra} of our nilpotent osculating groups.
The basic equality, 
\[ [fX,gY] = fg[X,Y] + f(X.g)Y - g(Y.f)X ,\]
shows that if $X, Y$ are sections of $H$ then {\em modulo} $H$ the value of the bracket $[X,Y](m)$ at $m\in M$ only depends on the values $X(m)$ and $Y(m)$ at $m$.
In other words, the commutator of vector fields induces a {\em pointwise} bracket,
\[ H_m \otimes H_m \to H_m \;\colon\; X\otimes Y \mapsto [X,Y]\,\text{\rm mod}\,H ,\]
where $m\in M$, and $N = TM/H$ denotes the quotient bundle.
This can be extended to a Lie bracket on $\Lg_m = H_m\oplus N_m$,
by taking $[\Lg_m,N_m] = 0$. 
Clearly, the Lie algebra $\Lg_m$ is two-step nilpotent.
The isomorphism class of these Lie algebras is a first infinitesimal invariant of the distribution $H$.
%In the context of the equivalence problem, the nilpotent {\em groups} don't play a role.
%(See \cite{Mo02} for a survey of more recent developments in this line of research.)

One can verify that the Lie algebras $\Lg_m$ are isomorphic to the Lie algebras of the osculating groups of the analysts.
However, this simple construction does not fully clarify the osculating structure. 
The analysis of hypoelliptic operators on a manifold $M$ depends on an approximation of differential operators on $M$ by translation invariant operators  on the nilpotent group.  
This requires an identification of an open subset of the nilpotent group with an open subset of the manifold,
and the identification has to  result in a `good' approximation of operators. 
The Lie algebras $\Lg_m$ are canonically identified with the fibers of the vector bundle $H\oplus N$,
which we may identify (by a choice of section $N\hra TM$) with the tangent bundle $TM$.
To identify a fiber of $TM$  with an open set in the manifold (locally) we need an {\em exponential map} $TM\to M$.
However, not every exponential map gives the desired degree of `closeness' of operators on the group and on the manifold.
To achieve a satisfactory  explanation of the formulas of the analysts we must clarify which exponential maps `osculate' the distribution $H$ on the manifold sufficiently closely, and precisely in what sense.

%We formulate a geometrically motivated notion of {\em $H$-adapted exponential map}, which is at once natural and subtle. The correct notion of $H$-adapted exponential map provides the crucial link between the Lie algebra bundle $H\oplus N$ and the osculating structures of the Heisenberg calculus. 
Before we can understand what the correct exponential maps are, we must clear up one final missing link.
The geometric literature provides an intrinsic definition of a bundle of nilpotent {\em Lie algebras},
which is derived from the Lie algebra of vector fields.
The analysis of hypoelliptic operators requires a system of osculating nilpotent {\em groups}.
The distinction turns out to be relevant.

\subsection{Nilpotent groups in sub-Riemannian geometry}

Nilpotent group structures also play a role  in {\em sub-Riemannian geometry}.
A sub-Riemannian manifold is a manifold $M$ together with a distribution $H\subset TM$ and a metric on $H$.
In sub-Riemannian geometry the tangent space $T_mM$ at a so-called `regular' point $m\in M$  carries the structure of a nilpotent group.
In the special case where $[H,H]=TM$ (brackets of vector fields in $H$ span $TM$ at each point in $M$) this group is isomorphic to the osculating group.
The group structure is defined by exponentiating the Lie algebra structure $\Lg_m=H_m\oplus N_m$ introduced in the previous section.
In sub-Riemannian geometry it is really the {\em group} structure on $T_mM$
that is of interest, not the Lie algebra structure.
It is the tangent space as a {\em group} that makes it a useful approximation to the manifold {\em as a sub-Riemannian metric space}.
The sense in which the geometry of the nilpotent group approximates the geometry of the sub-Riemannian manifold is closely related to the notion of `osculation' of operators as it appears in analysis.

In \cite[p.73--76]{Be96} Andr\'e Bella\"iche  considers the question, ``Why is the tangent space a group?''
Bella\"iche asks whether, in the context of sub-Riemannian geometry, there is a more {\em direct} definition of the group structure on $T_mM$ (or on $H_m\oplus N_m$).
Is there a way to define the group structure that does not depend on a prior definition of the Lie algebra?
Can we  explain geometrically how composition of group elements arises?
As we will show in this paper, it is precisely the formulation of a satisfactory answer to this question that  leads to the appropriate notion of `osculating' exponential maps, which, in turn, fully clarifies the intrinsic geometry underlying the osculating structures of the Heisenberg calculus.

It is interesting to note that Bella\"iche considers whether Alain Connes' {\em tangent groupoid} could contain a hint of how to answer his questions.
The tangent groupoid is obtained by taking the trivial groupoid $M\times M$ and `blowing up the diagonal'.
Composition of pairs in $M\times M$ is
\[
(a, b)\cdot (c, d) \begin{cases}
= (a, d) & \text{\rm if}\; b=c,\\
\text{not defined} & \text{\rm if}\; b\ne c.
\end{cases}
\] 
By introducing a parameter $t\in [0, 1]$,
we  let the pair $(a(t), b(t))$ converge to a tangent vector.
Provided that $a(0)=b(0)=m$ we have
\[ v = \lim_{t\to 0} \,\frac{a(t)-b(t)}{t} \in T_mM.\]
This defines a topology on the groupoid that is the union
\[ TM \cup M\times M\times (0,1].\]
As Connes shows, the tangent groupoid can be equipped with a natural smooth  structure. 
(See \cite{Co94}, II.5).

Connes' construction is of interest to Bela\"iche because it provides an intuitive construction of addition of tangent vectors (i.e.,  the group structure on the tangent space $TM$)
as the limit $t\to 0$ of the groupoid of pairs $M\times M$.
Bela\"iche does not believe that a similar construction could explain the nilpotent group structures arising in sub-Riemannian geometry. 
However, as we will see in the final section of this paper, there is a natural modification of Connes' tangent groupoid that is appropriate for the Heisenberg calculus (in section \ref{section:groupoid}),
and that exhibits the  nilpotent group structure of sub-Riemannian geometry as a limit of the pair groupoid $M\times M$. 
Our construction of this groupoid relies on the correct notion of `osculating' exponential maps.

\subsection{Overview of the paper}

In section \ref{section:arrows} we define a new kind of `tangent vector'---or, rather, a generalization of tangent vectors  appropriate for the definition of osculating groups.
We call these new objects `parabolic arrows'. 
%(We would have liked to call them ``parabolic vectors'', but that term would falsely suggest that they are a special type of vector.)
Like a tangent vector, a parabolic arrow is an infinitesimal approximation of a smooth curve near a point.
While tangent vectors are defined by means of first order derivatives,
parabolic arrows involve a mix of first and second order derivatives.

As we will see in section \ref{section:arrows}, parabolic arrows can be composed in  a natural way by extending them to local {\em flows}, which compose in the obvious way.
It is similar to the way ordinary tangent vectors could be added (by composition of flows), but the appearance of second order derivatives significantly complicates the picture.
In particular, composition of parabolic arrows is {\em noncommutative}.

Parabolic arrows at a point $m\in M$ are shown to form a nilpotent Lie group, and the Lie algebra $\Lg_m=H_m\oplus N_m$ (defined by taking brackets of vector fields) is shown to be its Lie algebra.
Thus, parabolic arrows are a geometric realization of elements in the osculating {\em group},
defined with reference to the geometry of curves and flows, instead of the usual definition via the Lie algebra of vector fields.

Algebraically, the osculating group structure arises from {\em automorphisms} of the algebra $C^\infty(M)$ (realized as diffeomorphisms of $M$) while the Lie algebra structure is derived from {\em derivations} of $C^\infty(M)$ (vector fields).
In section \ref{jet} we show that the formalism of jets of smooth maps $\RR\to \mathrm{Aut}\,(C^\infty(M))$ (representing flows) allows us to reformulate in a perhaps more conceptual way the definition of parabolic flows and parabolic arrows. 
At the same time, the jet formalism gives an easy generalization to the case of a filtration $H_1\subseteq H_2\subseteq ...\subseteq TM$ of subbundles of the tangent bundle, giving rise to several step nilpotent groups.

Armed with new geometric insight in the nature of the osculating group elements and their composition,
we introduce in section \ref{hypind10:section:exp} the appropriate notion of an exponential map for the fiber bundle of parabolic arrows.
Because parabolic arrows are related to curves on the manifold,
the group of parabolic arrows has a built-in relation to the manifold. 
Because of this connection a definition of `osculating' exponential map for parabolic arrows suggests itself naturally.
To show the effectiveness of our definition, we analyze the osculating structures of Folland and Stein for CR manifolds and of Beals and Greiner for Heisenberg manifolds, and we {\em derive} the explicit coordinate formulas used by these analysts from our  geometric concepts. 

Finally, in section \ref{section:groupoid} we will see that our construction of the osculating groups blends perfectly with the tangent groupoid formalism.
We employ parabolic arrows and osculating exponential maps to construct a tangent groupoid for the Heisenberg calculus.
We constructed  such a groupoid for the special case of contact manifolds in \cite{vE10a},
making use of Darboux's theorem.
In section \ref{section:groupoid} we will show how parabolic arrows and osculating exponential maps make the construction of a tangent groupoid for the Heisenberg calculus a straightforward generalization of Connes' construction.
(Our construction here applies in the case of an {\em arbitrary} distribution $H\subseteq TM$. In the special case of Heisenberg manifolds an alternative construction of this groupoid was given by Raphael Ponge in \cite{Po06}.)

\subsection{Simplified picture of the Heisenberg calculus}
 
To conclude this introduction, we sketch a simplified picture of the Heisenberg pseudodifferential calculus (and its generalizations) by means of the geometric concepts developed here.

Starting with a distribution $H\subseteq TM$
we have a filtration of the Lie algebra $\sm(TM)$ of vector fields on $M$,
where sections in $H$ are given order $1$ and all other sections in $TM$ have order $2$.
The associated {\em graded} Lie algebra can be identified with sections in the bundle of two-step graded Lie algebras $H\oplus N$, whose construction was explained above.
Let $T_HM$ denote the bundle of nilpotent Lie {\em groups} associated to $H\oplus N$,
which we identify with the fiber bundle of parabolic arrows.

%Choose a section $j\colon N\hra TM$ and the corresponding identification $j\colon H\oplus N\to TM$. We may then identify $T_HM$ with $TM$ as smooth fiber bundles.
%Also choose a connection $\nabla$ on $TM$.
% The crux of the matter is that the connection $\nabla$ must be chosen so as to {\em preserve the distribution $H$}.
% In other words, $\nabla$ must be a connection compatible with the $G$-structure on $TM$ that corresponds to the choice of a distribution $H\subseteq TM$.
%As we will see such an exponential map $\exp^\nabla\colon TM\to M$ induces an $H$-adapted exponential map 
Choose an $H$-adapted (`osculating') exponential map
\[ \exp^{\nabla}\;\colon\;T_HM\to M.\]
Extend this map to a local diffeomorphism near the zero section in $T_HM$,
\[ h\;\colon\; T_HM\to M\times M\;\colon\; (m,v)\mapsto (\exp_m^{\nabla}(v), m)\]
By means of the map $h$ we pull back the Schwartz kernel $k(m,m')$ of a continuous linear operator $C^\infty(M)\to C^\infty(M)$  to the bundle of osculating groups $T_HM$.
We obtain a smooth family $k_m(v)=k(h(m,v))$ of distributions on the nilpotent osculating groups $T_HM_m$ parametrized by $m\in M$.
If the operator associated to $k$ is pseudolocal, then each distribution $k_m$ has an isolated singularity at the origin of the group $T_HM_m$.
An operator with Schwartz kernel $k$ is in the calculus if each distribution $k_m$ has an asymptotic expansion (in homogeneous terms) on the nilpotent groups $T_HM_m$.

Precise details this approach to the generalized Heisenberg calculus are developed in a forthcoming publication \cite{vEY15}. 
The crucial simplification comes from the use of $H$-adapted (or `osculating') exponential maps introduced in this paper.

\vskip 6pt
\noindent{\bf Remark.} 
For simplicity of exposition, we focus in this paper on the geometry of osculating groups for manifolds equipped with a single distribution $H\subseteq TM$.
But the Heisenberg calculus has been generalized to arbitrary {\em filtered manifolds}.
A filtered manifold is a manifold with a nested sequence of distributions 
\[ H_1\subseteq H_2\subseteq \cdots \subseteq H_r=TM,\]
where it is required that the sections in these bundles form a filtration on the Lie algebra of vector fields, i.e.,
\[ [\sm(H_i), \sm(H_j)]\subseteq \sm(H_{i+j}).\]
In \cite{Cu89} Cummins develops the details of such a calculus for the case $r=3$, modeled on three-step nilpotent groups.
In an unpublished preprint of 1982  \cite{Me82}, Anders Melin develops a generalized Heisenberg calculus for  arbitrary filtered manifolds, modeled on graded nilpotent groups of arbirary length. 
Section \ref{jet} below constructs the group of generalized parabolic arrows that is suitable for this situation.
It will be shown in \cite{vEY15} how the geometric concepts developed in the current paper simplify the results of Melin.

\section{Parabolic Arrows and Their Composition.}\label{section:arrows}

\subsection{Parabolic arrows.}\label{sect:para}
Throughout this and the following sections, $M$ denotes a smooth manifold with a specified distribution $H\subseteq TM$.
We will write $N = TM/H$ for the quotient bundle,
and denote the fiber dimensions by $p=\text{dim\,}H$, $q=\text{dim\,}N$, and $n=p+q=\text{dim\,}M$.
We will {\em not} assume that $q=1$.

When studying a Heisenberg structure $(M,H)$ it is convenient to work with a special type of coordinates.

\begin{definition}
Let $m$ be a point on $M$, and $U\subseteq M$ an open set in $M$ containing $m$.
A coordinate chart $\phi\colon U \to \RR^n$, $\phi(m') = (x_1,\ldots,x_n)$ is called an {\em $H$-coordinate chart at $m$},
if $\phi(m) = 0$, and the first $p$ coordinate vectors $\partial/\partial x_i$ $(i=1,\ldots,p)$ at the point $m$ span the fiber $H_m$ of $H$ at $m$.
\end{definition}
Tangent vectors can be defined as equivalence classes of smooth curves.
By analogy, we introduce an equivalence relation involving second-order derivatives.

\begin{definition}
Let $c_1,c_2\colon [-1,1]\to M$ be two smooth curves that are tangent to $H$ at $t=0$.
For such curves we say that $c_1 \pareq c_2$ if $c_1(0)=c_2(0)$ 
and if, choosing $H$-coordinates centered at $c_1(0)=c_2(0)$, we have 
\begin{align*}
 c_1'(0) - c_2'(0) & = 0, \\
 c_1''(0) - c_2''(0) & \in H.
\end{align*}
An equivalence class $[c]_H$ is called a {\em parabolic arrow} at the point $c(0)$.
The set of parabolic arrows at $m\in M$ is denoted $T_HM_m$, while
\[ T_H M = \bigcup_{m\in M} T_HM_m .\]
\end{definition}
We can give $T_HM$ the topology induced by the $C^2$-topology on the set of curves, but for the moment we just think of $T_HM$ as a set.

\begin{lemma}\label{l1}
The equivalence relation $\sim_H$ is well-defined,
i.e., independent of the choice of the $H$-coordinates.
\end{lemma}
\begin{proof}
The condition that $c_1'(0) = c_2'(0)$ is clearly invariant.
We will show that, assuming $c_1'(0) = c_2'(0)$,
the condition 
$ c_1''(0) - c_2''(0) \in H$
on the second derivatives is independent of the choice of $H$-coordinates.
 
If $\psi$ is a change of $H$-coordinates, then:
\begin{align*}
 \frac{d^2 (\psi\circ c)}{dt^2} 
 & = \frac{d}{dt} \left(\sum_{j} \frac{\partial \psi}{\partial x_j}(c(t))  \frac{d c^j}{dt} \right) \\
 & = \sum_{j,k} \frac{\partial^2 \psi}{\partial x_j \partial x_k}(c(t)) \frac{d c^j}{dt} \frac{d c^k}{dt} \; + \sum_j \frac{\partial \psi}{\partial x_j}(c(t))  \frac{d^2 c^j}{dt^2}.
\end{align*}
At $t=0$ we assumed $d c_1/dt = d c_2/dt$, so that the first term on the right hand side is equal for $\psi\circ c_1$ and $\psi\circ c_2$ (at $t=0$). 
Therefore:
\[ (\psi\circ c_1)''(0) - (\psi\circ c_2)''(0) = 
	\frac{\partial \psi}{\partial x}(m)  \;\cdot\; \left( c_1''(0) - c_2''(0) \right). 
\]
Since $\psi$ is a change of $H$-coordinates at $m$, $\partial \psi/\partial x$ preserves $H_m$, so that $c_1''(0) - c_2''(0)\in H_m$ implies $(\psi\circ c_1)''(0) - (\psi\circ c_2)''(0) \in H_m$.

\end{proof}

If we fix $H$-coordinates at $m\in M$,
and consider the second-order expansion (in coordinates) of a curve $c$ with $c(0)=m$,
\[ c(t) = c'(0)\;t + \frac{1}{2} c''(0)\; t^2 + \OO(t^3) ,\]
we see that any such curve is equivalent, as a parabolic arrow, to a curve $\tilde{c}$ of the form
\[ \tilde{c}(t) = (th,t^2n) = (th_1,\ldots,th_p,t^2n_1, \ldots, t^2n_q).\]
This observation forms the basis for the following definition.

\begin{definition}\label{hypind10:def:taylorcoord}
Suppose we are given $H$-coordinates at $m\in M$.
Let $h \in \RR^p, n\in\RR^q$, 
and let $c(t)$ be the curve in $M$ defined (in $H$-coordinates) by
\[ c(t) = (th, t^2n) .\]
We call $(h,n) = (h_1,\ldots h_p,n_1,\ldots,n_q) \in \RR^{p+q}$
the {\em Taylor coordinates} for the parabolic arrow $[c]_H\in T_HM_m$,
induced by the given $H$-coordinates at $m$,
\[ F_m \;\colon\; \RR^{p+q}\to T_HM_m\; \colon\; (h,n)\mapsto [c]_H.\]
\end{definition}
This is analogous to the way in which coordinates on the tangent space $T_mM$ are induced by coordinates on $M$, with the important difference that Taylor coordinates on $T_HM_m$ are defined for only one fiber (i.e., one point $m\in M$) at a time. 

Analogous to the directed line segments that represent tangent vectors, a pictorial representation for the class $[c]_H$ would be a directed segment of a parabola.
Hence our name `parabolic arrow.' Parabolic arrows are what smooth curves look like infinitesimally, when we blow up the manifold using the dilations $(h,n)\mapsto (th,t^2n)$, and let $t\to \infty$.

When working with $H$-coordinates $\phi(m)=x\in \RR^n$, we use the notation $x = (x^H,x^N) \in \RR^{p+q}$, where
\[ x^H = (x_1,\ldots,x_p)\in \RR^p ,\; x^N = (x_{p+1},\ldots,x_{p+q})\in \RR^q.\]
\begin{lemma}\label{coord}
If $\psi$ is a change of $H$-coordinates at $m$, 
then the induced change of Taylor coordinates $\psi(h,n) = (h',n')$ 
for a given parabolic vector in $T_HM_m$ is given by the {\em quadratic} formula:
\begin{align*}
  h' & = D\psi(h) ,\\
  n' & = [D\psi(n) + D^2\psi (h,h)]^{N},
\end{align*}
where $[v]^N$ denotes the normal component of the vector $v=(v^H,v^N)\in \RR^{p+q}$.
\end{lemma}
\begin{proof}
This is just the formula for $(\psi\circ c)''(0)$ from the proof of Lemma \ref{l1}. 

\end{proof}

\begin{corollary}
The smooth structures on the set of parabolic arrows $T_HM_m$ at a point $m\in M$ defined by different Taylor coordinates are compatible, i.e., $T_HM_m$ has a natural structure of a smooth manifold.
\end{corollary}

It is clear from Lemma \ref{coord} that Taylor coordinates define a structure on $T_HM_m$ that is  more than just a smooth structure. This will be fully clarified when we introduce the {\em group} structure on $T_HM_m$, 
but part of this extra structure is captured if we consider how parabolic arrows behave when rescaled.
\begin{definition}
The family of dilations $\delta_s$, $s>0$, 
on the space of parabolic arrows $T_HM_m$
is defined by
\[ \delta_s([c]_H) = [c_s]_H ,\]
where $[c]_H$ is a parabolic arrow in $T_HM_m$ represented by the curve $c(t)$, and $c_s$ denotes the reparametrized curve $c_s(t) = c(st)$.
\end{definition}
When working in Taylor coordinates $[c]_H =(h,n)$, we simply have
\[ \delta_s(h,n) = (sh,s^2n) .\]
Clearly, these dilations are smooth maps and $\delta_{st}=\delta_s\circ \delta_t$.

Considering Taylor coordinates on $T_HM_m$, it is tempting to identify parabolic arrows with vectors in $H\oplus N$. Lemma \ref{coord} shows that such an identification is not invariant if we use Taylor coordinates to define it. 
But we have at least the following result.

\begin{lemma}\label{hypind11:lemma:liealg}
There is a {\em natural} identification
\[ T_0(T_HM_m) \cong H_m\oplus N_m \]
of the tangent space  $T_0(T_HM_m)$ at the `origin' (i.e., at the equivalence class $[0]_H$ of the constant curve at $m$) with the vector space $H_m\oplus N_m$.
It is obtained by identifying the coordinates on $T_0(T_HM_m)$ induced 
by Taylor coordinates on $T_HM_m$, with the natural coordinates on $H_m\oplus N_m$.
\end{lemma}
\begin{proof} 
From Lemma \ref{coord}, we see that Taylor coordinates on $T_0(T_HM_m)$ transform according to the formula
\[  h'  = D\psi(h) ,\, n' = D\psi(n)^{N},\]
because the quadratic term $D^2\psi(h,h)$ has derivative $0$ at $0$.
This is precisely how the induced coordinates on $H_m\oplus N_m$ behave under coordinate transformation $\psi$.

\end{proof}

\subsection{Composition of parabolic arrows.}\label{sect:pgrp}

We will now show that the manifold $T_HM_m$  has the structure of a Lie group.
Our method is based on composition of local flows of $M$.
By a flow $\Phi$ of $M$ we mean a diffeomorphism
$\Phi\colon M\times \RR\to M$, such that $m\mapsto \Phi_t(m)=\Phi(m,t)$ is a diffeomorphism for each $t\in \RR$, while $\Phi_0(m)=m$.
A {\em local flow} is only defined on an open subset $V\subseteq M\times \RR$.
Two flows $\Phi, \Psi$ can be composed:
\[ (\Phi\circ \Psi)(m,t) = \Phi(\Psi(m,t),\,t) .\]
Using the notation $\Phi_t$ for the local diffeomorphism $\Phi_t(m) = \Phi(m,t)$,
we can write $(\Phi\circ\Psi)_t = \Phi_t\circ \Psi_t$.

A (local) flow is said to be generated by the vector field $X\in \sm(TM)$, if
\[ \frac{\partial\Phi}{\partial t} \,(m,t) = X(m) .\]
However, we are specifically interested in  flows for which the generating vector field $X_t(m) =\frac{\partial\Phi}{\partial t}(m,t)$ is {\em not constant}, but depends on $t$.
We will only require that $X_0$ is a section in $H$, but we will allow $X_t$ to pick up a component in the $N$-direction.
This is because we are not primarily interested in the {\em tangent vectors} to the flow lines $c_m(t) = \Phi(m,t)$, but in the {\em parabolic arrows} that they define.

We start with a formula that gives a quadratic approximation (in $t$) for the composition of two arbitrary flows.

\begin{lemma}\label{lemm:l2}
Let $\Phi^X, \Phi^Y$ be two flows in $\RR^n$ that are defined near the origin,
and let $X$ and $Y$ be their generating vector fields at $t=0$:
\[ X(x)  = (\partial_t\Phi^X)(x,0) ,\;\text{and}\;Y(x)  = (\partial_t\Phi^Y)(x,0) .\]
Then the composition of $\Phi^X$ and $\Phi^Y$ has the following second-order approximation,
\[  (\Phi^X_t \circ \Phi^Y_t)(0) = \Phi^X_t(0) + \Phi^Y_t(0) + t^2\, (\nabla_{Y} X)(0) +\OO(t^3) ,\]
where $\nabla$ denotes the standard connection on $T\RR^n$.
\end{lemma}
\noindent{\bf Remark.} Observe that $X = \partial_t\Phi^X$ is required only at $t=0$!
\vskip 6pt
\begin{proof}
Write $F(r,s) = \Phi^X_r (\Phi^Y_s (0))$.
The Taylor series for $F$ gives
\begin{align*}
 F(t,t)  & = t \;\partial_r F(0,0) + t^2 \partial_s \partial_r F(0,0) + t \;\partial_s F(0,0) + \frac{1}{2} t^2 \partial_r^2 F(0,0) \\
         & \;\; + \frac{1}{2} t^2 \partial_s^2 F(0,0)    +\OO(t^3) \\
         & = F(t,0) + F(0,t) + t^2 \partial_s \partial_r F(0,0) +\OO(t^3),
\end{align*}        
or
\[ \Phi^X_t\Phi^Y_t(0) = \Phi^X_t(0) + \Phi^Y_t(0) + t^2 \left.\partial_s \partial_r \Phi^X_r \Phi^Y_s(0)\right|_{r=s=0} +\OO(t^3) .\] 
At r=0 we have $\partial_r \Phi^X_r = X$, so:
\[  \left.\partial_s \partial_r \Phi^X_r (\Phi^Y_s(0)) \right|_{r=0} = \partial_s\left( X(\Phi^Y_s(0))\right) .\] 
Here $X(\Phi^Y_s(0))$ denotes the vector field $X$ evaluated at the point $\Phi^Y_s(0)$, which can be thought of as a point on the curve $s \mapsto \Phi^Y_s(0)$.
The operator $\partial_s$ is applied to the components of this vector,
and the chain rule gives
\begin{align*}
 \left.\partial_s X(\Phi^Y_s(0)) \right|_{s=0}
                      & = \sum_{i=1}^p  \partial_i X(0)\cdot \left.\partial_s \Phi^Y_s(0)^i\right|_{s=0} \\
                      & = \sum_{i=1}^p  (\partial_i X)(0)\;Y^i(0)
                      = (\nabla_{Y} X)(0) 
\end{align*}
\end{proof}

We are interested in flows $\Phi$ for which the {\em flow lines} $\Phi^m(t) = \Phi(m,t)$ define parabolic arrows. Hence the following definition.

\begin{definition}
A {\em parabolic flow} of $(M,H)$
is a local flow $\Phi \colon V\to M$ (with $V$ an open subset in $M \times \RR$)
whose generating vector field at $t=0$,
\[ \frac{\partial\Phi}{\partial t} \,(m,0) \]
(defined at each point $m$ for which $(m,0)\in V$)
is a section of $H$.
\end{definition}
Given a {\em parabolic} flow $\Phi$, each of the flow lines $\Phi^m$ is tangent to $H$ at $t=0$, and so determines a parabolic arrow $[\Phi^m]_H$ at each $m\in M$ (with $(m,0)\in V$).
Once we have defined the smooth structure on $T_HM$ it will become clear that $m\mapsto \Phi^m$ is a smooth section of the bundle $T_HM$. It is an analogue of the notion of a generating vector field, but it is only defined at $t=0$.

We now show how composition of parabolic flows induces a group structure on the fibers of $T_HM$.

\begin{proposition}\label{prop:thm1}
Let $\Phi, \Psi$ be two {\em parabolic} flows. Then the composition
$(\Phi\circ \Psi)_t=\Phi_t\circ \Psi_t$
(defined on an appropriate domain) is also a parabolic flow, 
and the parabolic vector $[(\Phi\circ \Psi)^m]_H$ at a point $m\in M$ only depends on the parabolic vectors $[\Phi^m]_H$ and  $[\Psi^m]_H$ at the same point.
\end{proposition}
\begin{proof} 
Let $\nabla$ denote the standard local connection on $TM$ induced by the $H$-coordinates at $m$.
Because $\nabla_{fY}(gX) = fg\nabla_Y(X) + f(Y.g)X$, we see that the operation
\[ \Gamma^\infty(H) \otimes \Gamma^\infty(H) \to \Gamma^\infty(N) \;\colon\; 
                       (X,Y) \mapsto \left[ \nabla_YX \right]^N  \]
is $C^\infty(M)$-bilinear.
In other words, the $N$-component of $\nabla_YX$
at the point $m\in M$ only depends on the values  $X(m)$ and $Y(m)$ at $m$.
We denote this $N$-component by $\nabla^N$:
\begin{align*}
 & \nabla^N\;\colon\; H_m \otimes H_m \to N_m ,\\
 & \nabla^N(X(m),Y(m))= \left[ \nabla_YX \right]^N(m) .
\end{align*}
Lemma \ref{lemm:l2} implies:
\begin{align*}
 \Phi_t \Psi_t (0)^H & = \Phi_t(0)^H + \Psi_t(0)^H + \OO(t^2) ,\\ 
 \Phi_t \Psi_t (0)^N & = \Phi_t(0)^N + \Psi_t(0)^N + t^2 \nabla^N(X(0),Y(0)) +\OO(t^3).
\end{align*} 
Writing
\begin{align*}
 & \Phi_t(0)^H  = t h + \OO(t^2) ,\;  \Phi_t(0)^N  = t^2 n + \OO(t^3), \\
 & \Psi_t(0)^H  = t h' + \OO(t^2) ,\;  \Psi_t(0)^N  = t^2 n' + \OO(t^3),
\end{align*}
this becomes
\begin{align*}
 \Phi_t \Psi_t (0)^H & = t (h+h')   + \OO(t^2) ,\\ 
 \Phi_t \Psi_t (0)^N & = t^2 \left( n+n'+\nabla^N(h,h')\right) +\OO(t^3).
\end{align*} 
The proposition is a direct corollary of these formulas.

\end{proof}

It is clear from Proposition \ref{prop:thm1} that composition of parabolic flows induces a group structure on the set $T_HM_m$, for each $m\in M$, analogous to addition of tangent vectors in $T_mM$.
To see that $T_HM_m$ is actually a {\em Lie group}, we use the explicit formulas obtained in the proof of Proposition  \ref{prop:thm1} .

\begin{proposition}\label{prop:thm2}
Let $\Phi, \Psi$ be two parabolic flows.
Given $H$-coordinates at $m$, let $X_i\in \Gamma(H)$ $(i=1,\ldots,p)$ be local sections in $H$
that extend the coordinate tangent  vectors $\partial_i$ at $m$.
Let $X_i^l$ $(l=1,\ldots,n)$ denote the coefficients of the vector field $X_i$, i.e., 
\[ X_i = \sum X_i^l\partial_l .\]
Let $(b_{ij}^k)$ be the array of constants
\[ b_{ij}^k = \partial_j X_i^{p+k}(m) \]
for $i,j=1,\ldots,p$ and $k=1\ldots,q$.
It represents a bilinear map $b\colon\RR^p\times\RR^p\to\RR^q$ via
\[ b(v,w)^k =  \sum_{i,j=1}^p b_{ij}^k v^i w^j ,\]
with $k=1,\ldots,q$.

If $(h,n)$ and $(h',n')$ are the Taylor coordinates of $[\Phi^m]_q$ and  $[\Psi^m]_q$, respectively, then the Taylor coordinates $(h'',n'')$ of $[(\Phi\circ \Psi)_m]_q$ are given by
\begin{align*}
 h'' & = h+h' , \\ 
 n'' & = n+n'+b(h,h').
\end{align*} 

\end{proposition}
\begin{proof}
This is a direct corollary of the formulas in the proof of Proposition \ref{prop:thm1}.
Simply observe that
\[ b(\partial_i,\partial_j)^k = \partial_j X_i^{p+k}(m) = [\nabla_{X_j} X_i]^{p+k}(m) = \nabla(\partial_i,\partial_j)^{p+k} ,\]
which implies that $b(v,w) = \nabla^N(v,w)$.

\end{proof}

\begin{corollary}\label{coro:gdef}
The operation
\[ [\Phi^m]_H \ast [\Psi^m]_H  = [(\Phi\circ \Psi)^m]_H \]
defines the structure of a Lie group on $T_HM_m$.

\end{corollary}
The groups $T_HM_m$ are the {\em osculating groups}
associated to the distribution $(M,H)$.
Note that, although the value of the array $(b_{ij}^k)$ in Proposition \ref{prop:thm2} depends on the choice of coordinates, our definition of the group elements (parabolic vectors) as well as their composition is clearly coordinate-independent.
Furthermore, it is important to notice that the values of $b_{ij}^k$ do {\em not} depend on the choice of vector fields $X_i$,
but only on the choice of coordinates.

Beals and Greiner {\em defined} the osculating groups by means of the formulas we have derived in Proposition \ref{prop:thm2} (see \cite{BG88}, chapter 1).  Note that in their treatment $q=1$, so that the $k$-index in the array $b^k_{ij}$ is missing. The osculating group itself was simply identified with the coordinate space $\RR^n$, and its status as an independent geometric object was left obscure.

\begin{proposition}
The natural dilations $\delta_s$ of the osculating groups $T_HM_m$, induced by reparametrization of curves, are Lie group automorphisms.
\end{proposition}
\begin{proof}
That the dilations are group automorphisms follows immediately from the geometric definition of the group operation on $T_HM_m$ in Corollary \ref{coro:gdef} (by reparametrizing the flows).
Alternatively, using Taylor coordinates we have
$ \delta_s(h,n) = (sh,s^2n) $,
which is clearly a smooth automorphism for the group operation
\[ (h,n,)\ast(h',n') = (h+h',n+n'+b(h,h')) .\] 

\end{proof}

The construction of the osculating bundle $T_HM$ is {\em functorial} for (local) diffeomorphisms. Given a diffeomorphism of manifolds with distribution,
\[ \phi \;\colon\; (M,H)\to (M',H'),\]
such that $D\phi\colon H \to H'$,
we could define the {\em parabolic derivative} $T_H\phi$ of $\phi$ as the map
\[ T_H\phi\;\colon\; T_HM\to T_HM'\;\colon\; [c]_H\mapsto[\phi\circ c]_H ,\]
where $c(t)$ is a curve in $M$ representing a parabolic arrow $[c]_H\in T_HM_m$.
A straightforward calculation, similar to the proof of Lemma \ref{l1}, shows that this is a well-defined map (independent of the choice of the curve $c$),
and functoriality is obvious, i.e.,
\[ T_H(\phi\circ\phi') = T_H\phi \circ T_H\phi'.\]
Clearly, if $\phi$ is a diffeomorphism, then $T_H\phi$ is a group isomorphism in each fiber.
%The same is true for {\em immersions}.
% At the level of the osculating Lie algebras, the logarithm of the parabolic derivative is simply the obvious map induced by $D\phi$,
% \[  D_H\phi\;\colon\; H\oplus N\to H'\oplus N' \;\colon\; (m,h,n)\mapsto (m,D\phi_m(h), D\phi_m(n)\,{\rm mod}\, H') .\]
% Since the bracket operation for vector fields commutes with the derivative $D\phi$, the map $D_H\phi$ is a Lie-algebra homomorphism in each fiber. Consequently, the parabolic derivative $T_H\phi$ is a group homomorpism when restricted to each osculating group. In other words, $T_H$ is a functor from the category of smooth manifolds with Heisenberg structure to the category of {\em smooth groupoids} (see section \ref{hypind30:section:groupoids}).

\subsection{The Lie algebras of the osculating groups.}

According to Lemma \ref{hypind11:lemma:liealg}, we may identify the Lie algebra ${\rm Lie\,}(T_HM_m)$ as a vector space with $H_m\oplus N_m$.
In the introduction we defined a Lie algebra structure on $H_m\oplus N_m$, and we now show that it is compatible with the group structure on $T_HM_m$.
We make use of some general results on two-step nilpotent groups that are discussed in the appendix. 

\begin{proposition}\label{prop:lie}
Let $X$ and $Y$ be two (local) sections of $H$.  
Then the value of the normal component $[X,Y]^N(m)$ of the bracket $[X,Y]$ at the point $m$ only depends on the values of $X$ and $Y$ at the point $m$. 

The Lie algebra structure on 
${\rm Lie\,}T_HM_m \cong H_m\oplus N_m$ is given by
\[ [(h,n),(h',n')] = \left(0,[X,Y]^N(m)\right) ,\]
where $X,Y\in\sm(H)$ are arbitrary vector fields with $X(m) = h, Y(m) = h'$.
\end{proposition}
\begin{proof}
This is a straightforward application of Lemma \ref{gl2} to the group structure on $T_HM_m$
as described in Proposition \ref{prop:thm1}. We have
\[ b(h',h)-b(h,h') = \left(\nabla_XY - \nabla_YX\right)^N(m) = [X,Y]^N(m) .\] 
We have already shown that $b(h,h') = \left(\nabla_YX\right)^N(m)$ only depends on $h=X(m)$ and $h'=Y(m)$.

\end{proof}

\noindent We are now in a position to define the smooth structure on the total space 
\[ T_HM = \bigcup T_HM_m .\]
There is a natural bijection
\[ \exp \;\colon\; H_m\oplus N_m \to T_HM_m ,\]
namely the exponential map from the Lie algebra $H_m\oplus N_m$ to the Lie group $T_HM_m$.
We give the total space $T_HM$ the smooth structure that it derives from its identification with $H\oplus N$.

\begin{lemma}
The smooth structure on $T_HM$, obtained by the fiberwise identification with $H\oplus N$ via  exponential maps, is compatible with the Taylor coordinates on each $T_HM_m$, for any choice of $H$-coordinates at $m$.
\end{lemma}
\begin{proof} 
Choosing $H$-coordinates at $m$, we get linear coordinates on $H_m\oplus N_m$.
Taking these coordinates and Taylor coordinates on $T_HM_m$,
we have identified ${\rm Lie\,}T_HM_m \cong H_m\oplus N_m$.
According to Proposition \ref{prop:gp1}, the exponential map $H_m\oplus N_m \to T_HM_m$ is expressed in these coordinates as
\[ \exp(h,n) = (h,n+\frac{1}{2}b(h,h)) ,\]
which is clearly a diffeomorphism.

\end{proof}

%A {\em graded Lie algebra}  $\Lg = \Lg_r \oplus \cdots \oplus \Lg_1$ (where elements in $\Lg_i$ are of degree $i$) has a Lie bracket that is compatible with the grading,
%\[ [\Lg_i,\Lg_j] \subseteq \Lg_{i+j} ,\]
%with $[\Lg_i,\Lg_j] = \{0\}$ if $i+j>r$. Dilations $\delta_t$ associated to the grading are linear maps defined by
% \[ \delta_t(X) = t^iX,\;\text{for}\, X\in \Lg_i .\] 
%Clearly, $\delta_t,t>0$ is a one-parameter group of Lie algebra automorphisms.

The natural decomposition ${\rm Lie\,}(T_HM_m) = H_m\oplus N_m$
defines a Lie algebra grading, with $\Lg_1 = H_m$ of degree 1 and  $\Lg_2 = N_m$ of degree $2$.
Corresponding to the grading we have dilations $\delta_t(h,n) = (th,t^2n)$,
and these dilations are Lie algebra automorphisms.
The dilations of the osculating group $T_HM_m$ induced by reparametrization of curves and the dilations of the graded Lie algebra $H_m\oplus N_m$ are related via the exponential map (see Proposition \ref{prop:gp1}):
\begin{align*}
 \exp(\delta_t(h,n)) &= \exp(th, t^2n) = (th,t^2n+\frac{1}{2}b(th,th)) \\
 & = \delta_t(h,n+\frac{1}{2}b(h,h)) = \delta_t\exp(h,n) 
 \end{align*}
It will be useful to characterize the parabolic arrows whose logarithms are vectors in $H$.

\begin{proposition}\label{prop:hug}
If $c\colon \RR \to M$ is a curve such that $c'(t)\in H$ for all $t\in (-\varepsilon,\varepsilon)$, then the parabolic arrow $[c]_H\in T_HM_m$ is the exponential of the tangent vector $c'(0)\in H_m$, considered as an element in the Lie algebra of $T_HM_m$. Here $m=c(0)$.
\end{proposition} 
\begin{proof}
Choose $H$-coordinates at $m=c(0)$,
and let $(h,n)\in \RR^{p+q}$ be the corresponding Taylor coordinates of the parabolic arrow $[c]_H$.
Because $c'(t)\in H$ for $t$ near $0$, 
we can choose an $H$-frame $X_1,\ldots,X_n$ in a neighborhood $U$ of $m$ in such a way that
$c'(t) = \sum h_i X_i(c(t))$ at every point $c(t)\in U$.
With this set up, we compute the second derivative:
\[ \frac{d^2 c}{dt^2} = \frac{d}{dt} (\sum_i h_i X_i)\circ c
	= \sum_j \frac{ \partial}{\partial x_j} \left( \sum_i h_i X_i \right) \frac{d c_j}{dt}
	= \sum_{i,j} h_i\,\frac{ \partial X_j}{\partial x_i} \,\frac{d c_j}{dt}
.\]
Then, at $t = 0$, the normal component of $c''(0)$ is given by,
\[ c''(0)^N = \sum_{i,j} \partial_i X_j^N  h_i h_j = b(h,h), \]
%\[ \frac{d^2 c}{dt^2}	= \sum_{i,j} \frac{ \partial X_j}{\partial x_i} v_i v_j  =B(v,v), \]
where $b(h,h)$ is defined as in Proposition \ref{prop:thm2}.
It follows that the Taylor coordinates of $[c]_H$ are $(h,\frac{1}{2}b(h,h))$,
and therefore, by Proposition \ref{prop:gp1},
\[ \log([c]_H) = (h,0) .\]

\end{proof}

\section{ A jet point of view on the osculating group}\label{jet}

In this section we want to introduce a more conceptual point of view on  parabolic flows, considered as maps from the real line to the group of diffeomorphisms of  the manifold $M$.  The construction of the Hesenberg tangent group $T_HM$ involves the 2-jets of such maps. In order to allow generalization to the case of filtrations $H_1\subset H_2\subset ...\subset TM$, we shall consider jets of all orders. For the convenience of the reader, we will at first replace the group of diffeomorphisms of $M$ by a Lie group $G$.
 
\subsection { Jets with values in a Lie group}

 Let $G$ be a Lie group and $\EuFrak{g}$ be its Lie algebra. We denote by $U(\EuFrak{g})$ the universal enveloping algebra of $\EuFrak{g}$. 
 Recall that $U(\EuFrak{g})$ is also a bialgebra with coproduct
 $$\Delta : U(\EuFrak{g})\mapsto U(\EuFrak{g})\otimes U(\EuFrak{g})$$
 satisfying $\Delta (X)=X\otimes 1+1\otimes X$ for any $X\in\EuFrak{g}$.

 Let us define the group $J^{\infty}G$ of jets as follows: We consider the group of all $C^{\infty}$ maps from the real line ${\bf R}$ to $G$ which map $0$ to the identity of the group $G$ (with pointwise multiplication) and we quotient by the normal subgroup consisting of maps having zero Taylor series, i.e. maps $t\mapsto g(t)$ such that for any $f\in C^{\infty}(G)$, the function $t\mapsto f(g(t))$ has derivatives of all orders vanishing at 0.

 The jet of the map  $t\mapsto g(t)$ is determined by the Taylor series of the functions $t\mapsto f(g(t))=f(1)+tT_1f(1)+t^2T_2f(1)+...+t^kT_kf(1)+...$ where $T_k\in U(\EuFrak{g})$ is the left invariant differential operator on $G$ defined by $T_k(f)(1)=(f\circ g)^{(k)}(0)/k!$, the $k$-th derivative at zero of  the function $(f\circ g)(t)=f(g(t))$. 

A jet in $J^{\infty}G$ is therefore a formal power series $j_t=1+tT_1+t^2T_2+...+t^kT_k+...$ with coefficients $T_k\in U(\EuFrak{g})$. Moreover the elements $T_k$ of $ U(\EuFrak{g})$ must satisfy the condition of being group-like elements: indeed the obvious fact that $(f_1f_2)(g(t))=f_1(g(t))f_2(g(t))$ implies that $\Delta (j_t)=j_t\otimes j_t$, i.e.
$$\Delta (T_k)=\sum_{i+j=k}T_i\otimes T_j.$$

In particular, $\Delta (T_1)=T_1\otimes 1+1\otimes T_1$, which means that $T_1$ is an element of $\EuFrak{g}$.

Conversely any formal series $j_t=1+tT_1+t^2T_2+...+t^kT_k+...$ with coefficients $T_k\in U(\EuFrak{g})$ satisfying the above conditions, is the jet of a smooth map from the real line to $G$. This can be proved using the Lie algebra and the exponnential map. 
The Lie algebra of $J^{\infty}G$ is indeed he space $\EuFrak {j}^{\infty}G$ of series $tX_1+t^2X_2+...+t^kX_k+...$ with $X_k\in\EuFrak{g}$ for all $k\geq 1$. The exponential map ${\rm exp}: \EuFrak {j}^{\infty}G\rightarrow J^{\infty}G$ is easily checked to be  bijective, and sends the Lie algebra to the group-like elements.

In particular the exponential map gives the solution to the equations  $\Delta (T_k)=\sum_{i+j=k}T_i\otimes T_j$ defining $J^{\infty}G$, namely 
$$T_k=\sum_{l=1}^k{1\over l!}\sum_{k_1+k_2+...+k_l=k}X_{k_1}X_{k_2}...X_{k_l}$$
with $X_1, X_2,... $ is a sequence of elements of the Lie algebra $\EuFrak{g}$.

\subsection {The filtrated case}

Let us now assume that we are given a filtration
  $F$ of the Lie algebra 
 $\EuFrak{g}$, i.e. a sequence of vector subspaces
 $F_1\subset F_2\subset ...\subset F_l\subset \EuFrak{g}$
 such that $[F_i,F_j]\subset F_{i+j}$.
(Note that we define $F_k=\EuFrak{g}$ for all $k\geq l+1$.) As well known, to such a filtered Lie algebra one can associate a graded Lie algebra:
$${\rm gr}({\EuFrak g}, F)=F_1\oplus F_2/F_1\oplus F_3/F_2... \oplus \EuFrak{g}/F_l$$
with the obviously defined brackets $[F_i/F_{i-1},F_j/F_{j-1}]\subset F_{i+j}/F_{i+j-1}$.

A simple and conceptual way of looking at the associated graded algebra is as follows. 
 Consider the Lie subalgebra $\EuFrak {j}^{\infty}_FG$ of $\EuFrak {j}^{\infty}G$, consisting of series $tX_1+t^2X_2+...+t^kX_k+...$  with $X_k\in F_k$ for all $k\geq 1$. 
The Lie algebra $\EuFrak {j}^{\infty}_F(G)$ has an obvious ideal, namely its image by multiplication by $t$: the ideal $\EuFrak {i}^{\infty}_FG=t\EuFrak {j}^{\infty}_FG$ is the set of series $t^2X_1+t^3X_2+...+t^{k+1}X_k+...$  with $X_k\in F_k$ for all $k\geq 1$. 

Then the quotient Lie algebra $\EuFrak {j}^{\infty}_FG/\EuFrak {i}^{\infty}_FG$ is  clearly isomorphic to the graded Lie algebra ${\rm gr}({\EuFrak g}, F)=F_1\oplus F_2/F_1\oplus ...\oplus \EuFrak{g}/F_l$ associated to the filtered algebra $(\EuFrak{g}, F)$.

The above formalism allows to give a simple interpretation of the Lie group associated to the Lie algebra ${\rm gr}({\EuFrak g}, F)$. Indeed, the exponential maps the Lie algebra $\EuFrak {j}^{\infty}_FG$ to a subgroup $J^{\infty}_FG$ of $J^{\infty}G$, which is the set of jets $1+tT_1+t^2T_2+...+t^kT_k+...$ where
$$T_k=\sum_{l=1}^k{1\over l!}\sum_{k_1+k_2+...+k_l=k}X_{k_1}X_{k_2}...X_{k_l}$$
with $X_k\in F_k$ for all $k\geq 1$. The image by the exponential of the ideal $\EuFrak {i}^{\infty}_FG$ is a normal subgroup $I^{\infty}_FG$ of $J^{\infty}_FG$. The quotient  group 
$$J^{\infty}_FG/I^{\infty}_FG$$
is a Lie group with Lie algebra $\EuFrak {j}^{\infty}_FG/\EuFrak {i}^{\infty}_FG={\rm gr}({\EuFrak g}, F)$.

\subsection{Jets of flows on a manifold}

To deal with flows on a manifold $M$, we are interested in extending the above formalism to the case where $G$ is the group of diffeomorphisms of a manifold $M$ ( in fact an infinite dimensional Lie group). The associated Lie algebra $\EuFrak{g}$ is the Lie algebra $\Gamma (TM)$ of vector fields on $M$ (with the Lie bracket) and the enveloping algebra of $\EuFrak{g}$ is the algebra of differential operators on $M$ (which map constants to constants, i.e., whose zero order component is a constant). Note that $\EuFrak{g}=\Gamma (TM)$ is not only a Lie algebra but also a $C^{\infty}(M)$ module,  which is projective and of finite type, and compatible with the action of $\EuFrak{g}$ on $C^{\infty}(M)$ by derivations: for any $X, Y\in\Gamma (TM)$, $f,g\in C^{\infty}(M)$, $[fX,gY]=fg[X,Y]+fX(g)Y-gY(f)X$. 

The group $J^{\infty}G$ is the group of jets of flows which are identity at $t=0$. An element of $J^{\infty}G$ is a formal series $j_t=1+tT_1+t^2T_2+...+t^kT_k+...$ where each $T_k$ is a differential operator of order $k$, and satisfying the equations 
$$\Delta (T_k)=\sum_{i+j=k}T_i\otimes T_j$$
which concretely means that for any smooth functions $f_1$ and $f_2$ on $M$ 
$$T_k(f_1f_2)=\sum_{i+j=k}T_i(f_1)T_j(f_2).$$

The Lie algebra of $J^{\infty}G$ is the space $\EuFrak {j}^{\infty}G$ of series $tX_1+t^2X_2+...+t^kX_k+...$ with $X_k$ are vector fields  for all $k\geq 1$.  The exponential map describes the jets of flows as 
given by $$T_k=\sum_{l=1}^k{1\over l!}\sum_{k_1+k_2+...+k_l=k}X_{k_1}X_{k_2}...X_{k_l}$$
with $X_1, X_2,... $  a sequence of vector fields.
In other words a jet of flows is given (via the exponential map) by coordinates which are vector fields $X_1$, $X_2$,...

\subsection{Jets of flows on a manifold equipped with a filtration of the tangent bundle}

Let us now consider a  filtration $F$ of the Lie algebra of vector fields, i.e. $F_1\subset F_2\subset ...\subset F_l\subset\EuFrak{g}$ satisfying the following two conditions:

 1) $F_j$ is a sub-$C^{\infty}(M)$-module of the module of sections of the tangent bundle $TM$;
 
 2) $[F_i,F_j]\subset F_{i+j}$ (with $F_k$ consists of all vector fields for all $i\geq l+1$). 
 
 The first condition says that the filtration is defined pointwise on each fibers. For simplicity we shall restrict ourselves to the case where $F_j=\Gamma (H_j)$ is the module of sections of a subbundle $H_j$ of $TM$ with $H_1\subset H_2\subset ...\subset H_l\subset TM$ is a filtration of the tangent space subject to the conditions $[\Gamma (F_i),\Gamma (F_j)]\subset\Gamma (F_{i+j})$.
But more general cases are allowed.
The second condition is meaningless in the case where $l=1$, i.e., for a single subbundle $H\subset TM$. But for $l\geq 2$ the pointwise filtration is not enough and a condition on the brackets of vector field is needed. 
 
 As is well known, to such a filtration is associated a graded  Lie algebra  
 $${\rm gr}({\EuFrak g}, F)=F_1\oplus F_2/F_1\oplus F_3/F_2... \oplus \EuFrak{g}/F_l$$
with brackets $[F_i/F_{i-1},F_j/F_{j-1}]\subset F_{i+j}/F_{i+j-1}$. 
An important fact is that the above graded Lie algebra is a bundle equipped with a Lie algebra structure. Indeed, each quotient $F_j/F_{j-1}$ is a $C^{\infty}(M)$-module and the Lie bracket $F_i/F_{i-1}\times F_j/F_{j-1}\rightarrow F_{i+j}/F_{i+j-1}$ is  $C^{\infty}(M)$ bilinear, as easily checcked using the formula $[fX,gY]=fg[X,Y]+fX(g)Y-gY(f)X$. This shows that each fiber  $H_1\oplus H_2/H_1\oplus ...\oplus TM/H_l$ is equipped with a Lie algebra structure.

As above, we can associate to such a filtration a Lie subalgebra of $\EuFrak {j}^{\infty}(G)$, namely the set $\EuFrak {j}^{\infty}_FG$ of series $tX_1+t^2X_2+...+t^kX_k+...$  with $X_k\in F_k$ for all $k\geq 1$. 
The associate infinite dimensional Lie group  $J^{\infty}_FG$ is the group of jets of flows $1+tT_1+t^2T_2+...+t^kT_k+...$ which are of the form 
$$T_k=\sum_{l=1}^k{1\over l!}\sum_{k_1+k_2+...+k_l=k}X_{k_1}X_{k_2}...X_{k_l}$$
with $X_k\in F_k$ for all $k\geq 1$. 

The Lie algebra $\EuFrak {j}^{\infty}_FG$ has an ideal $\EuFrak {i}^{\infty}_FG=t\EuFrak {j}^{\infty}_FG$, consisting of series $t^2X_1+t^3X_2+...+t^{k+1}X_k+...$  with $X_k\in F_k$ for all $k\geq 1$. The quotient Lie algebra $\EuFrak {j}^{\infty}_FG/\EuFrak {i}^{\infty}_FG$ is the above considered graded Lie algebra ${\rm gr}({\EuFrak g}, F)=F_1\oplus F_2/F_1\oplus ...\oplus \EuFrak{g}/F_l$ associated to the filtered algebra $(\EuFrak{g}, F)$.

We propose to call the elements of $J^{\infty}_FG$ {\it jets of generalized parabolic flows}. 
The quotient group $J^{\infty}_FG/I^{\infty}_FG$ is the quotient of the group of generalized parabolic flows by the normal subgroup consisting of flows as above with $X_1=0$ and $X_j\in F_{j-1}$. The fiber at point $x\in M$ is the quotient of the group of generalized parabolic flows by the normal subgroup consisting of flows such that $X_1=0$ and $X_j\in F_{j-1}$ at point $x$. It might be called the {\it group of generalized parabolic arrows} at point $x\in M$.

\subsection{Back to the special case of one single subbundle}

Let us now explain how in the case of a filtration $H\subset TM$ by a single subbundle, we recover the above definition of parabolic flows. The condition for a jet of flow $j_t=1+tT_1+t^2T_2+...+t^kT_k+...$ to be in the subgroup $J^{\infty}_ FG$ is simply that $T_1=X_1\in\Gamma (H)$. This is what we called a (jet of) parabolic flow. A parabolic flow is given by a 2-jet $\exp (tX+t^2Y)=1+tX+{t^2\over 2}(X^2+2Y)$ where $X$ is a section of $H$ and $Y$ any vector field. The product of two such flows $(1+tX_1+{t^2\over 2}(X_1^2+2Y_1))(1+tX_2+{t^2\over 2}(X_2^2+2Y_2))$ is calculated as follows (with the rule $t^3=0$) being 

$$1+t(X_1+X_2)+{t^2\over 2}(X_1^2+X_2^2+2X_1X_2+2Y_1+2Y_2)$$
which is the same as:
$$1+t(X_1+X_2)+{t^2\over 2}((X_1+X_2)^2+2(Y_1+Y_2+{1\over 2}[X_1,X_2]))$$

Now the subgroup of parabolic jets of the form $\exp (t^2Y)=1+{t^2\over 2}Y$ where $Y$ is a section of $H$ is  normal, and the quotient is clearly the Heisenberg tangent group $T_HM=H\oplus TM/H$.

\section{$H$-adapted Exponential Maps.}\label{hypind10:section:exp}

\subsection{Exponential maps for parabolic arrows}

We now come to the definition of {\em exponential map} that is suitable for the  fiber bundle of parabolic arrows.
We will see that such an exponential map is equivalent to what analysts mean by an osculating structure. 
Our Definition  \ref{hypind10:def:exp2} of $H$-adapted exponential maps clarifies the intrinsic geometry behind the explicit coordinate formulas for osculating structures proposed by analysts.
\vskip 6pt

Before introducing our notion of $H$-adapted exponential map, let us first state  exactly what we mean by `exponential map'.

\begin{definition}\label{hypind10:def:exp}
An {\em exponential map} is a smooth map
\[\exp\colon TM\to M\]
whose restriction $\exp_m\colon T_mM\to M$ to a fiber $T_mM$ maps $0\mapsto m$, while the derivative $ D\exp_m$ at the origin $0\in T_mM$ is the identity map $T_0(T_mM)=T_mM \to T_mM$. 
\end{definition}
A more specialized notion of `exponential map'  associates it to a connection $\nabla$ on $TM$.
For a vector $v\in T_mM$ one can define $\exp^\nabla_m(v)$ to be the end point $c(1)$ of the unique curve $c(t)$ that satisfies
\[ c(0)=m,\; c'(0)=v ,\; \nabla_{c'(t)}c'(t) = 0.\]
Such a map certainly satisfies the property of Definition \ref{hypind10:def:exp}.
One could specialize further and let $\nabla$ be the Levi-Civita connection for a Riemannian metric on $M$, in which case the curve $c$ will be a geodesic.
But the looser Definition \ref{hypind10:def:exp} suffices for the purpose of defining a pseudodifferential calculus on $M$.
Given such an exponential map, we can form a diffeomorphism $h$ in a neighborhood of the zero section $M\subset TM$,
\[ h\;\colon\; TM\to M\times M\;\colon\; (m,v)\mapsto (\exp_m(v), m).\]
Given the Schwartz kernel $k(m,m')$ of a linear operator $C^\infty(M)\to \mathcal{D}(M)$, we can pull it back to a distribution on $TM$ by means of the map $h$.
The singularities of the pullback $k\circ h$ reside on the zero section $M\in TM$.
The classical pseudodifferential calculus is obtained by specifying the asymptotic expansion of $k\circ h$ near the zero section.
In order to develop the Heisenberg pseudodifferential calculus in an analogous manner we must identify the natural notion of exponential map for the fiber bundle $T_HM$ of parabolic arrows.
Our geometric insight in the group structure of $T_HM$ leads to a natural definition.
Observe that Definition \ref{hypind10:def:exp} is equivalent to the condition that each curve $c(t)=\exp_m(tv)$ has tangent vector $c'(0)=v$ at point $m=c(0)$.
This immediately suggests the following generalization.

\begin{definition}\label{hypind10:def:exp2}
Let $M$ be a manifold with distribution $H\subset TM$. 
An {\em $H$-adapted exponential map} for $(M,H)$ is a smooth map
\[ \exp \;\colon\; T_HM \to M ,\]
such that for each parabolic arrow $v\in T_HM_m$ the curve $c(t)=\exp_m(\delta_tv)$ in $M$ represents the parabolic arrow $v$, i.e., $[c]_H=v$ in $T_HM_m$.
\end{definition}
If we choose a section $j\colon N\hra TM$ we may identify $H\oplus N$ with the tangent bundle $TM$. The composition
\[ T_HM \stackrel{\log}{\longrightarrow} H\oplus N \stackrel{j}{\longrightarrow} TM ,\]
then identifies $T_HM$, as a smooth fiber bundle, with $TM$.
Every $H$-adapted exponential map is identified, in this way, with an ordinary exponential map $\exp\colon TM\to M$.
However, not every exponential map $TM\to M$ induces an $H$-adapted exponential map.
Definition \ref{hypind10:def:exp} involves only the {\em first derivative} of the map. 
But in order to be $H$-adapted, an exponential map must satisfy a further requirement on its {\em second derivative}.
Definition \ref{hypind10:def:exp2} provides a natural geometric way to encode this rather delicate second order condition, by means of our notion of parabolic arrows.

The key property that makes an arbitrary exponential map $H$-adapted can be stripped down to the condition that every curve $c(t) = \exp(th)$, for every $h\in H_m$, represents the parabolic arrow $[c]_H=h\in H_m$.
This, in turn, is equivalent to the requirement that there exists a second curve $c_2(t)$ that is everywhere tangent to $H$ and such that $c_2'(0)=c'(0)$ and such that $c_2''(0)$ agrees with $c''(0)$ in the directions transversally to $H_m$
(Proposition \ref{prop:hug}).
We are not sure if this makes things any clearer, but it does bring out, to some extent, the geometric meaning of $H$-adaptedness.
The point is that, while the bundle $H$ may not be integrable,
one still wants the exponential map to be such that {\em rays} in $H$ are mapped to curves in $M$ that `osculate' the bundle $H$ as closely as possible,
as measured by the {\em second} derivative in the transversal direction.

\subsection{$H$-adapted exponential maps and connections}

Further geometric insight in the distinctive features of $H$-adapted exponential maps is obtained if we consider exponential maps that arise from connections. 
This consideration is also useful because it implies the {\em existence} of $H$-adapted exponential maps.

Observe that the choice of a distribution $H\subseteq TM$ is equivalent to a reduction of the principal frame bundle $fTM$ to the subbundle $fT_HM$, whose fiber at $m\in M$ consists of frames $(e_1,\ldots,e_n)$ in $T_mM$ for which $(e_1,\ldots, e_p)$ is a frame in $H_m$. 
(A local section of $fT_HM$ is what we have called an $H$-frame.)  
The fiber bundle $fT_HM$ is a principal bundle with structure group,
\[ G = \{ \left(\begin{array}{cc} A & B \\ 0 & C    \end{array}\right) \in {\rm End}\,(\RR^p\oplus \RR^q) \} \subseteq GL_{p+q}(\RR) .\]  
In other words, a Heisenberg structure on $M$ is equivalent to a $G$-structure on $TM$, and the natural connections to consider are connections on the principal $G$-bundle $fT_HM$.
For the associated affine connection $\nabla$  on $TM$, this simply means that if $X\in \sm(H)$, then $\nabla_Y X \in \sm(H)$, in other words, $\nabla$ is a connection on $TM$ that can be restricted to a connection on $H$.
One easily verifies that this last condition implies that the exponential map $\exp^\nabla\colon TM\to M$ is $H$-adapted.
We abuse the notation $\exp^\nabla$ for the induced exponential map on parabolic arrows,
\[ \exp^\nabla\;\colon\; T_HM\to M,\]
even though this also involves choosing an explicit section $j\colon N\to TM$.
In practice one may choose to work only with $H$-adapted exponential maps induced by $G$-connections.
This approach has the advantage that it does not require the notion of parabolic arrows.
However, it is insufficient to encompass the exisiting definitions of osculating structures found in the literature.
We will see, for example, that the osculating structures of Beals and Greiner, while they do arise from $H$-adapted exponential maps, do not arise from $G$-connections.

\subsection{$H$-adapted exponential maps from $H$-coordinates}

A second method for constructing $H$-adapted exponential maps is by means of a smooth system of $H$-coordinates. 
A smooth system of $H$-coordinates is a choice of $H$-coordinates 
\[ E_m\;\colon\; \RR^n\to M ,\]
at each point $m\in M$, in such a way that the map $(m,x)\mapsto E_m(x)$ is smooth.
Observe that, in general, such a smooth system will only exist {\em locally},
and the exponential maps defined by means of such a system of coordinates are likewise only locally defined. 
Such a local definition suffices for the purpose of specifying the Heisenberg calculus.

Let
\[ F_m\;\colon\; \RR^n\to T_HM_m \]
denote the Taylor coordinates induced by the $H$-coordinates $E_m$.
It is immediately clear from the definition of Taylor coordinates (Definition \ref{hypind10:def:taylorcoord}) that the composition
\[ \exp_m=E_m\circ F_m^{-1} \;\colon\; T_HM_m \to M \]
is an $H$-adapted exponential map.

This particular way of obtaining an $H$-adapted exponential map, while geometrically more clumsy, is most useful in explaining existing definitions of osculating group structures in the literature.
The map $E_m$ identifies the coordinate space $\RR^n$ with a neighborhood of a point $m$ in the manifold $M$, while $F_m$ identifies the same coordinate space with the group $T_HM_m$.
The typical procedure of analysts is to specify explicit formulas for a group structure on the coordinate space $\RR^n$.
But in each case their procedure is clarified if we identify their formulas as specific instances of our Taylor coordinates $F_m$.
We will analyze two exemplary cases.

\vskip 6pt
{\bf The osculating structures of Folland and Stein.}
Osculating structures on contact manifolds first appeared in the work of Folland and Stein (see \cite{FS74}, sections 13 and 14).
The construction of Folland and stein can be summarized as follows (we generalize slightly).

On a $(2k+1)$-dimensional contact manifold, with $2k$-dimensional bundle $H\subseteq TM$,
choose a (local) $H$-frame $X_1,\ldots,X_{2k+1}$.
It can be shown that this frame can be chosen such that
\begin{align*}
  [X_i,X_{k+i}] & = X_{2k+1} \,{\rm mod}\; H,\; \text{\rm for}\,i=1,\ldots, p, \\
  [X_i,X_j] & = 0 \,{\rm mod}\; H,\; \text{\rm for all other values of}\, i\le j .
\end{align*}
Then, for $v\in \RR^{2k+1}$, let $E_m(v)$ be the endpoint $c(1)$ of the integral curve $c(t)$ of the vector field $\sum v_iX_i$ with $c(0) = m$, in other words,
\[ E_m\;\colon\; \RR^{2k+1}\to M\;\colon\; v\mapsto  \Phi^1_{\sum v_iX_i}\,(m) ,\]
where $\Phi^t_Y$ denotes the flow generated by a vector field $Y$.
For Folland and Stein, the `osculating Heisenberg structure' on $M$ is the family of maps $E_m$, identifying an open subset of the Heisenberg group $H_k = \RR^{2k+1}$ (with its standard coordinates)  with a neighborhood of $m\in M$.

In the framework we have established here, we see that the commutator relations for the $H$-frame allow us to identify the basis $X_i(m)\in H_m\oplus N_m$ of the osculating Lie algebra with the standard basis of the Lie algebra of the Heisenberg group.
Accordingly, we have a group isomorphism,
\begin{align*}
 F_m\;\colon\; H_k=\RR^{2k+1}&\to H_m\oplus N_m          \to T_HM_m \\
                                                   v &\mapsto   \sum v_i X_i(m) \mapsto \exp(\sum v_i X_i(m)) 
 \end{align*}
 We recognize $F_m$ as the Taylor coordinates on $T_HM_m$ for the $H$-coordinates $E_m$ at $m$. 
We see that the osculating structure on $M$, as defined by Folland and Stein,
can be interpreted as the $H$-adapted exponential map $\exp_m = E_m\circ F_m^{-1}$.

Alternatively, we can conceptualize the construction of Folland and Stein as specifying an $H$-adapted exponential map by means of a (locally defined) {\em flat} $G$-connection $\nabla$ on $TM$.
The connection $\nabla$ is the one for which the $H$-frame $X_1,\ldots,X_n$ is parallel.

\vskip 6pt
{\bf The Heisenberg manifolds of Beals and Greiner.}
Our second example is the group structure defined by Beals and Greiner on the coordinate space for given $H$-coordinates $E_m\colon \RR^n\to U$ at a point $m$ (see \cite{BG88}, section 1.1).
Beals and Greiner only consider the case where $H\subset TM$ is an arbitrary distribution of codimension one.
Proposition \ref{prop:thm2} provides an explicit formula for composition in the osculating group $T_HM_m$ in terms of the Taylor coordinates $F_m\colon \RR^n\cong T_HM_m$. 
This formula, when specialized to the case $k=1$, agrees exactly with the unexplained string of formulas (1.8), (1.11), (1.15) in \cite{BG88}.
Those formulas can thus be reinterpreted quite simply as the coordinate expression of an $H$-adapted exponential map $\exp_m=E_m\circ F_m^{-1}$.

\vskip 6pt
If we compare these two examples, we see that Beals and Greiner start their construction with an arbitrary system of $H$-coordinates $E_m$, and are therefore required to compensate by a quadratic correction term in the Taylor coordinates $F_m$.
Folland and Stein, on the other hand, choose a very specific type of coordinates $E_m$ that is better suited to the Heisenberg structure, and as a result obtain Taylor coordinates $F_m$ that are simply the linear coordinates on $H_m\oplus N_m$.    
Both strategies can be used in the case of more general distributions $H\subseteq TM$.

\section{The Parabolic Tangent Groupoid}\label{section:groupoid}

In \cite{Co94} section II.5, Connes introduces the tangent groupoid as part of a streamlined proof of the Atiyah-Singer index theorem.
The convolution algebra of this groupoid encodes the quantization of symbols (as functions on $T^*M$) to classical pseudodifferential operators.
The smoothness of the tangent groupoid proves, in an elegant geometric way, that the principal symbol of a classical pseudodifferential operator is invariantly defined as a distribution on $TM$ (or, equivalently, its Fourier transform on $T^*M$).

In this section we discuss the definition of a tangent groupoid that plays the same role for the Heisenberg calculus.
Our notion of parabolic arrows immediately suggests a topology for this groupoid.
But our proof that the groupoid has a well-defined smooth structure relies most crucially on the notion of $H$-adapted exponential maps.
The succesful definition of a tangent groupoid by means of $H$-adapted maps can be considered as proof that the principal symbol in the Heisenberg calculus is well-defined.
More specifically, it proves that we can use an arbitrary $H$-exponential map to pull back Schwartz kernels from $M\times M$ to $T_HM$ when we develop the calculus.

\subsection{The parabolic tangent groupoid and its topology}\label{section:THM}

As a generalization of Connes' tangent groupoid, which relates the total space of the tangent bundle $TM$ to the pair groupoid $M\times M$, we define a similar groupoid in which the bundle of abelian groups $TM$ is replaced by the fiber bundle of osculating groups $T_HM$.
We shall refer to this groupoid as the {\em parabolic tangent groupoid} of a manifold with distribution $H\subseteq TM$, and denote it by $\THM$.

As an algebraic groupoid, $\THM$ is the disjoint union,
\[ \THM = (\bigcup_{t\in (0,1]}\GG_t) \cup (\bigcup_{m\in M} \GG_m) ,\]
of a parametrized family of pair groupoids with the collection of osculating groups,
\begin{align*}
 \GG_t = M\times M,\; t\in (0,1],\\
 \GG_m = T_HM_m,\; m\in M .
\end{align*}
Clearly, the union $\cup \GG_t = M\times M\times (0,1]$ by itself is a smooth groupoid,
and the same is true, as we have seen, for the bundle of osculating groups $\cup \GG_m = T_HM$. 
We write $\GG_0 = T_HM$, and $\GG_{(0,1]}=M\times M\times (0,1]$.
Each groupoid $\GG_t,t\in [0,1]$ has object space $M$, 
and the object space for the total groupoid $\GG = \THM$ is the manifold, 
\[ \GG^{(0)} = M \times [0,1] .\]
We will endow $\THM$ with the structure of a manifold with boundary, by glueing $\GG_0$ as the $t=0$ boundary to $\GG_{(0,1]}$.
The topology on $\THM$ is such that $\GG_{(0,1]}$ is an open subset of $\THM$.
The topology on $\THM$ can be defined very nicely by means of our parabolic arrows. 
The construction in \cite[II.5]{Co94} generalizes immediately if we replace tangent vectors by parabolic arrows.

\begin{definition}\label{def:top}
A curve $(a(t),b(t),t)$ in $\GG_{(0,1]}=M\times M\times (0,1]$ converges, as $t\to 0$,
to a parabolic arrow $(m,v) \in T_HM$ if,
\begin{align*}
 M \ni m & = \lim_{t\to 0} a(t) = \lim_{t\to 0} b(t) , \\
 T_HM_m \ni v & =  [a]_H \ast [b]_H^{-1} ,
\end{align*} 
where we assume that $a'(0)\in H$ and $b'(0)\in H$. 
\end{definition}
Recall that $[a]_H$ and $[b]_H$ denote the parabolic arrows defined by the curves $a, b$, while the expression $[a]_H\ast[b]_H^{-1}$ denotes the product of $[a]_H$ with the inverse of $[b]_H$ in the osculating group $T_HM_m$.
If $H=TM$, the osculating groups are abelian, and the definition simplifies to
\[ T_mM \ni v = a'(0) - b'(0) = \lim_{t\to 0} \frac{a(t)-b(t)}{t} .\]
This is precisely the topology of the tangent groupoid as defined by Connes in \cite{Co94}.

It is easy to see that the groupoid operations for $\THM$ are continuous.
For example, in $\GG_{(0,1]}$ we have,
\[ (a(t),b(t),t)\;\cdot\;(b(t),c(t),t) = (a(t),c(t),t) ,\]
while in $\GG_0$,
\[ ([a]_H \ast [b]_H^{-1}) \,\ast\, ([b]_H \ast [c]_H^{-1}) = [a]_H \ast [c]_H^{-1} ,\] 
assuming that $a(0)=b(0)=c(0)$ and $a'(0),b'(0),c'(0)\in H_m$.
However, we will not rigorously develop this point of view. 
Instead, we glue $\GG_0$ to $\GG_{(0,1]}$ in an alternative way, more convenient for practical use, by defining a smooth structure on $\THM$.

\subsection{Charts on the parabolic tangent groupoid}

We now define charts on the groupoid $\THM$ by means of exponential maps,
completely analogous to the construction in \cite[II.5]{Co94}.

Suppose we have an $H$-adapted exponential map,
\[ \exp\;\colon\; T_HM\to M .\]
We define a map, 
\[ \psi\;\colon\;T_HM\times [0,1) \to \THM, \]
by,
\begin{align*}
 \psi(m,v,t) & = (\exp_m(\delta_tv),m,t), \;{\rm for}\, t>0 ,\\
 \psi(m,v,0) & = (m,v) \in T_HM_m .
\end{align*}
Here $\delta_t$ denotes the Heisenberg dilation in the osculating group $T_HM_m$.
The smooth structure on $T_HM\times [0,1)$ induces a smooth structure in an open neigborhood of $\GG_0=T_HM$ in $\THM$.

At first sight it may seem that the main modification to the construction of the classical tangent groupoid is the introduction of the parabolic dilations $\delta_tv$ to replace the simple `blow-up' $tv$ of Connes.
This simple idea is indeed the obvious thing to do if one wants to define a tangent groupoid for the Heisenberg calculus.
But the real crux of the definition of the parabolic tangent groupoid is the requirement that the exponential map must be $H$-adapted.
If one uses arbitrary exponential maps to define the charts on $\THM$, in combination with the parabolic dilations $\delta_t$, then the resulting charts are {\em not} smoothly compatible (i.e., transition functions would not be smooth).
It is precisely for this reason that arbitrary exponential maps cannot be used to define Heisenberg pseudodifferential operators.

\vskip 6pt
To better understand the glueing of the bundle $T_HM$ to $M\times M\times (0,1]$ it may be helpful to recall the construction of an $H$-adapted exponential map by means of a system of $H$-coordinates (in an open set $U\subseteq M$),
\[ E_m \;\colon\; \RR^n\to U .\]
Recall that $H$-coordinates are such that, for each $m\in U$, the coordinate vectors $dE_m(\partial/\partial x_i)$, for $i=1,\ldots,p$, are vectors in $H_m$.
As before, let
\[ F_m \;\colon\; \RR^{p+q}\to T_HM_m\]
denote the Taylor coordinates on the osculating group $T_HM_m$, induced by the $H$-coordinates $E_m$ (Definition \ref{hypind10:def:taylorcoord}). 
As was discussed, the composition $\exp_m(v)=E_m\circ F_m^{-1}$ is a (local) $H$-adapted exponential map.
From this perspective, equivalent to the chart $\psi$ that was defined above we could work with the chart,
\[ \psi'\;\colon\;U\times \RR^p\times \RR^q \times [0,1)  \to \THM, \]
defined by,
\begin{align*}
 \psi'(m,h,n,t) & = (E_m(th,t^2n),m,t), \;{\rm for}\, t>0 ,\\
 \psi'(m,h,n,0) & = F_m(m,h,n) \in T_HM_m .
\end{align*}
This description brings out very clearly what is going on.
The expression $(E_m(th,t^2n),m,t)$ for $t>0$ corresponds to a `blow up' of the diagonal in $M\times M$ by a factor $t^{-1}$ in the direction of $H$, and by a factor $t^{-2}$ in the direction  transversal to $H$. 
This is precisely what one would expect.

But the success of the construction crucially depends on the choice of coordinates at $t=0$, involving the Taylor coordinates $F_m$.
And this is the subtle ingredient in the construction of the groupoid.
Recall that, if we make the {\em canonical} identification of $T_HM$ with the bundle $H\oplus N$ (by means of the Lie exponential map in the fibers), then the Taylor coordinates $F_m$ are explicitly given by,
\[ \log\,F_m(m,h,n) = (h,n-\tfrac{1}{2}b_m(h,h)) \in H_m\oplus N_m ,\]
where $b_m(h,h)$ is a quadratic form that depends on the coordinates $E_m$ (see Propositions \ref{prop:thm2} and \ref{prop:gp1}).
The necessity and nature of this quadratic correction term $b_m(h,h)$ would be very hard to guess if one had to construct the parabolic tangent groupoid from scratch.

If the coordinates $E_m$ are chosen in such a way that the corresponding bilinear form $b_m$ is skew-symmetric (for example, as in the construction of Folland and Stein),
then this quadratic correction term vanishes, and we can work simply with the natural coordinates on $H_m\oplus N_m$ at $t=0$.
Correspondingly, an alternative solution to the construction of the parabolic tangent groupoid would be to work with `preferred' coordinate systems $E_m$, i.e., $H$-coordinates for which $b_m$ is skew-symmetric.

\subsection{Proof that the smooth structure is well-defined}

We now show that, with the above choices, the manifold structure on $\THM$ is well defined. 
The proofs of Propositions \ref{hypind40:prop:manifoldTHM} and \ref{prop:sect2-cont} show the relevance of the corrected groupoid coordinates at $t=0$, if arbitrary $H$-coordinates $E_m$ are allowed.
The basic ingredient of the proof is the following technical lemma.
\begin{lemma}\label{hypind40:lemma:technical}
Let $\phi\colon T_HM\to T_HM$ be a diffeomorphism that preserves the fibers; fixes the zero section $M\subset T_HM$; and at the point $m$ has derivative $D\phi_m={\rm id}$, and a second derivative that satisfies $D^2\phi_m(h,h)\in H_m$, for $h\in H_m$.
%Moreover, for each $(m,v)\in T_HM$, the curves $a(t)=\phi(m,\delta_t v)$ and $b(t)=\delta_tv$ satisfy,
%\begin{align*}
% a'(0) - b'(0) & =0,\\
% a''(0)-b''(0) & \in H_m.
%\end{align*}
Then the map,
\[  \widetilde{\phi} \;\colon\; T_HM\times \RR\  \to T_HM\times \RR ,\]
defined by,
\begin{align*}  
  \widetilde{\phi}(m,v,t) & =  (\delta_t^{-1}\phi(m,\delta_t v), t) ,\\
  \widetilde{\phi}(m,v,0) & =  (m,v,0),
\end{align*}
is a diffeomorphism.
\end{lemma}
\begin{proof}
Clearly, $\widetilde{\phi}$ is smooth on the open subset where $t\ne 0$.
We must prove that $\widetilde{\phi}$ is smooth in a neighborhood of the $t=0$ fiber.

For convenience of notation, we identify $T_HM$ with $H\oplus N$ via the logarithm.
The proof is based on a simple Taylor expansion near $t=0$. 
For a choice of coordinates on $H\oplus N$ %that are linear in the fibers, 
we have,
\[ \phi(m,v) = \phi(m,0) +D\phi_m(v)+\frac{1}{2}D^2\phi_m(v,v)+R(m,v) .\]
The remainder term $R=R(m,v)$ satisfies a bound $|R|<C|v|^3$, for $|v|<1$.
Now write $v=h+n$ with $h \in H_m, n\in N_m$. Then,
\begin{align*}
  &\phi(m,th+t^2n) \\
  & = \phi(m,0) +tD\phi_m(h)+t^2D\phi_m(n) \\ 
    & +\frac{1}{2}t^2D^2\phi_m(h,h)+ t^3D^2\phi_m(h,n) + \frac{1}{2}t^4D^2\phi_m(n,n) + R(m,th+t^2n)\\
    & =  \phi(m,0) + tD\phi_m(h)+ t^2D\phi_m(n)
    + \frac{1}{2}t^2D^2\phi_m(h,h)+ t^3R'.
\end{align*}    
The error term $R'=r'(m,h,n,t)$, satisfies a bound $|R'|\le C$ for $|h|<\epsilon |t|^{-1}$, $|n|<\epsilon |t|^{-2}$. 
Observe that these inequalities hold in an open neighborhood of the $t=0$ fiber in $T_HM\times \RR$.
 
The assumptions on $\phi$ allow the simplification,
\[ \phi(m,\delta_tv)  =  (m,\, th+ t^2n+\frac{1}{2}t^2D^2\phi_m(h,h)+t^3R') ,\]
where $D^2\phi_m(h,h)\in H_m$. We find,
\[ \delta_t^{-1}\phi(m,\delta_t v)  =  (m,\, v+tR'') ,\]
where, again, the coefficient of the remainder $R''$ is uniformly bounded in a neighborhood of the $t=0$ fiber.
This implies continuity of $\widetilde{\phi}$.

By the same reasoning, expanding $\phi$ in a higher order Taylor series, one obtains, 
\[ \widetilde{\phi}(m,v,t) = (m,\, v+ \sum_{k=1}^r a_k t^k + R_r t^r,t) ,\]
where the coefficients $a_k=a_k(m,v)$ are smooth functions, independent of $t$, arising from the derivatives of $\phi$, while the coefficient $R_r$ of the remainder is uniformly bounded in a neighborhood of $t=0$.
This implies smoothness of $\widetilde{\phi}$.
  
\end{proof}

\begin{proposition}\label{hypind40:prop:manifoldTHM}
For different choices of $H$-adapted exponential maps 
 the charts $\psi\colon T_HM\times [0,1]\to \THM$, defined above, have smooth transition functions.
In other words, $\THM$ has a well-defined structure of smooth manifold, independent of the choice of $H$-adapted exponential map.
\end{proposition}
\begin{proof}
Let $\psi$ and $\psi'$ be the two maps,
\[ \psi,\psi'\;\colon\; T_HM\times [0,1] \to \THM,\]
constructed in the manner explained above, for two different exponential maps
$E, E'\;\colon\; T_HM\to M$.
We must prove that the transition function $\widetilde{\phi}=\psi^{-1}\circ \psi'$ is smooth.
We have,
\begin{align*}
 \widetilde{\phi}(m,v,t) & = (\delta_t^{-1}E_m^{-1}(E'_m(\delta_tv)),m,t), \;{\rm for}\, t\ne 0 ,\\
 \widetilde{\phi}(m,v,0) & = (m,v,0) .
\end{align*}
Definition \ref{hypind10:def:exp2} of $H$-adapted exponential maps immediately implies that the composition $\phi_m= E_m^{-1}\circ E'_m$ satisfies the assumptions of 
Lemma \ref{hypind40:lemma:technical}. Hence, $\widetilde{\phi}$ is smooth.

\end{proof}

\subsection{Compatibility of smooth structure and topology}

The next proposition shows that the manifold structure on $\THM$ is compatible with the topology according to Definition \ref{def:top}.

\begin{proposition}\label{prop:sect2-cont}
Suppose $a(t),b(t)$ are smooth curves in $M$ with $a(0)=b(0)=m$,
such that $a'(0)$ and $b'(0)$ are in $H_m$.
Then in $\THM$, endowed with the manifold structure defined above,
\[ \lim_{t\to 0}\, (a(t),b(t),t) = [a]_H\ast[b]_H^{-1} \in T_HM_{m} .\]
\end{proposition}
\noindent{\it First proof.}
If we {\em assume} that the curve $(a(t),b(t),t)$ in $\GG_{(0,1]}$ extends to a smooth curve in $\THM$,
then there is a nice proof that makes use of parabolic flows.
Let $v_0\in T_HM_m$ be the point in $\GG_0$ to which the curve in $\GG_{(0,1]}$ converges,
and let $v_t$ be the parabolic arrow defined by,
$ \psi(b(t),v_t,t) = (a(t),b(t),t) $, i.e.,
\[ a(t) = \exp_{b(t)}(\delta_tv_t) .\]
By definition of the manifold structure on $\THM$, we have $(b(t),v_t)\to (m,v_0)$.
We see that the section $v_t, t\in [0,1]$ is smooth along $b(t)$, and can be extended to a section $V$ in a neigborhood of $m=b(0)$. 
Now define a flow,
\[ \Phi^t_v(m')= \exp_{m'}(\delta_t V(m')).\]
By definition of Heisenberg exponential maps, the curve,
\[ t\mapsto \exp_{m'}(\delta_t V(m')) \]
has parabolic arrow $V(m')$. In other words, $\Phi_v^t$ is a parabolic flow, and in particular,
\[ [\Phi^t_v(m)]_H = V(m)=v_0 \in T_HM_m.\]
Clearly $a(t) = \Phi_v^t (b(t))$. 
Extend $b(t)$ to a parabolic flow $\Phi^t_b$, such that $b(t)=\Phi_b^t(m)$. Then we see that,
\[ [a]_H = [\Phi_v^t\circ\Phi_b^t(m)]_H = [\Phi_v^t(m)]_H\ast [\Phi^t_b(m)]_H = v_0 \ast [b]_H,\]
which means that $v_0=[a]_H\ast [b]_H^{-1}$.

\hfill $\Box$

\vskip 6 pt
\noindent{\it Second proof.}
To prove the proposition without the extra assumption of convergence, and to illustrate a different technique, we give a second proof.

We use the map $\psi'$ defined above to describe the manifold structure on $\THM$.
We need a system of $H$-coordinates $E_m$, and the corresponding Taylor coordinates $F_m$.
Let us identify an open set $U\subseteq M$ with $\RR^n$ (via a coordinate map that we suppress in the notation).
Given an $H$-frame $X_i$ on $U$, we have a  system of coordinates,
\[ E_m\;\colon\; \RR^n\to U \;\colon\; v\mapsto m+\sum v_i X_i(m) = m+Xv .\]
Here $X=(X_i^j)$ denotes the $n\times n$ matrix whose columns are the vector-values functions $X_i\colon U\to \RR^n$.

Expand $a$ and $b$ in the coordinates on $U\cong \RR^n$ as,
\begin{align*}
 a(t) & = th+t^2k +\OO(t^3) ,\\
 b(t) & = th'+t^2k' +\OO(t^3) ,
\end{align*}
assuming that $a(0)=b(0)=0$.
We have Taylor coordinates,
\[ F_m(h,n) = [a]_H,\; F_m(h',n') =[b]_H,\]
where $n=k^N$, $n'=k'^N$ are the normal components of $k,k'$.
With the notation of Proposition \ref{prop:thm2}, we  compute,
\begin{align*}  
   F_m^{-1}([a]_H\ast [b]_H^{-1}) & = (h,n) \ast (h',n')^{-1} \\ & = (h,n) \ast (-h',-n'+b(h',h')) \\ &  = (h-h',n-n'-b(h,h')+b(h',h')).       
\end{align*}
Now let $(a(t),b(t),t) = \psi'(b(t),x(t),y(t),t)$, i.e.,
\[  a(t) = E_{b(t)}(tx(t),t^2y(t)) ,\]
where the coordinates $(x(t),y(t))\in \RR^{p+q}$ depend on $t$.
We must show that,
\[  \lim_{t\to 0} (x(t),y(t)) = (-h+h',-n+n'-b(h,h')+b(h,h)) .\]
We approximate the coordinates $(x(t),y(t))\in \RR^{p+q}$ by a Taylor expansion of $E_{b(t)}^{-1}(a(t))$, using the explicit form of $E_m$, as follows, 
\begin{align*}
 &(tx(t),t^2y(t))  = X_{b(t)}^{-1}\left(a(t)-b(t)\right) \\
                     & = a(t)-b(t)  + t\,D (X^{-1})_0\left(h',a(t)-b(t)\right) + \OO(t^3) \\
                     & = t(h-h') + t^2(k-k') + t^2\,D(X^{-1})_0(h',h-h') + \OO(t^3)
\end{align*}
Let us explain the calculation.
In the first step we expanded $X^{-1}(b(t))$. Because $a(t)-b(t)=\OO(t)$, it sufficed to consider only the first derivative,
\[ \frac{\partial}{\partial t} \left.X^{-1}(b(t))\right|_{t=0}  =D(X^{-1})_0.h'.\]
In the second step we expanded $a(t)-b(t)$, again ignoring terms of order $\OO(t^3)$.

Reversing the dilation, we find,
\[ (x(t),y(t)) =  \left(h-h', n-n' + D(X^{-1})^N_0(h',h-h')\right) \;+ \OO(t).\]
Because $X_0=1$, we have $D(X^{-1})_0 = -D X_0$, while the normal component  $D X^N_0(h',h-h')$ is equal to $b(h',h-h')$, by definition of the bilinear form $b$.
This gives the desired result. 

\hfill $\Box$

\section*{Appendix: Two-step nilpotent groups.}\label{sect:nilp}

We collect here some simple facts about two-step nilpotent groups that play a role in this paper. These facts are elementary, but we are not aware of a reference that contains the formulas we need.

Recall that a  Lie algebra $\Lg$ is called {\em two-step nilpotent} 
if $[[\Lg,\Lg],\Lg] = 0$.
The Campbell--Baker--Hausdorff formula for such Lie algebras has very few non-zero terms:
\[ \exp{(x)}\, \cdot \, \exp{(y)} = \exp{(x + y + \frac{1}{2} [x,y] )} ,\]
for $x,y\in  \Lg$.
Replacing the bracket $[x,y]$ with an arbitrary (not necessarily skew-symmetric) bilinear map 
$B\;\colon\; \RR^p\times \RR^p\to\RR^q$,
we can generalize and define a Lie group $G_B = \RR^p\times\RR^q$ with group operation
\[ (h_1,n_1) \,\ast\, (h_2,n_2) = (h_1+h_2, n_1+n_2+ B(h_1,h_2)) .\]
It is trivial to verify the group axioms (using the bilinearity of $B$).
By Proposition \ref{prop:thm2}, the group structure of parabolic arrows $T_HM_m$ expressed in Taylor coordinates is of this type. 
Our main goal in this section is to prove the following proposition.

\begin{proposition}\label{prop:gp1}
Let $G_B$ be the Lie group defined above. With the natural coordinates on $G_B = \RR^{p+q}$ and ${\rm Lie}\,{G_B} = T_0\RR^{p+q}$,
 the exponential map $\exp\ \colon {\rm Lie\,}{G_B}  \to G_B$ is expressed as
\[ \exp(h,n) = (h,n+\frac{1}{2}B(h,h)) .\]
\end{proposition}
The proof consists of a string of lemmas.

\begin{lemma}\label{gl2}
The Lie algebra structure on ${\rm Lie}\, G_B = \RR^{p+q}$
is given by the bracket
\[ [(h_1,n_1),(h_2,n_2)] = \left(0,\;B(h_1,h_2)- B(h_2,h_1)\right).\]
In particular, the Lie algebra structure only depends on the skew-symmetric part $(B-B^T)/2$ of the bilinear map $B$.
\end{lemma}
\begin{proof}
The neutral element in $G_B$ is $(0,0)$, and inverses are given by 
\[ (h,n)^{-1} = (-h,-n+B(h,h)).\]
Commutators in $G_B$ are calculated as follows:
\begin{align*}
 & (h_1,n_1)\ast(h_2,n_2)\ast(h_1,n_1)^{-1}\ast(h_2,n_2)^{-1} \\
 & = (h_1,n_1)\ast(h_2,n_2)\ast(-h_1,-n_1+B(h_1,h_1))\ast(-h_2,-n_2+B(h_2,h_2)) \\
 & = (h_1+h_2,n_1+n_2+B(h_1,h_2)) \ast\\
 & \quad\quad (-h_1-h_2,-n_1-n_2+B(h_1,h_1)+B(h_2,h_2)+B(-h_1,-h_2)) \\
 & = (0,B(h_1,h_1)+B(h_2,h_2)+2B(h_1,h_2)+B(h_1+h_2,-h_1-h_2)) \\ 
 & = (0,B(h_1,h_2)-B(h_2,h_1)) .
\end{align*}
Replace $(h_i,n_i)$ with $(th_i,tn_i)$ and take the limit as $t \to 0$.

\end{proof}

We see that the groups $G_B$ are indeed two-step nilpotent, or even abelian in the trivial case where $B$ is symmetric. 

\begin{lemma}\label{gl3}
If $B\colon \RR^p\times \RR^p\to\RR^q$ is a {\em skew-symmetric} bilinear map,
then the exponential map $\exp\colon {\rm Lie\,}(G_B) \to G_B$
is the usual identification of $T_0\RR^n$ with $\RR^n$.
\end{lemma}
\begin{proof}
For any $(h,n)\in \RR^{p+q}$ we have $(th,tn)\ast(sh,sn) = ((t+s)h,(t+s)n)$. In other words, the map
\[ \phi\;\colon\;\RR \to G_b\;\colon\; t\mapsto (th,tn) \]
is a group homomorphism. 
The tangent vector to this one-parameter subgroup at $t=0$ is $\phi'(0) = (h,n)\in {\rm Lie\,}(G_B)$,
and by definition $\exp(\phi'(0)) = \phi(1) = (h,n)\in G_B$.

\end{proof}

\begin{lemma}\label{gl1}
If $B,C\colon \RR^p\times \RR^p\to\RR^q$ are two bilinear maps
that have the same skew-symmetric part, then the quadratic map
\[ \phi\;\colon\;G_C \stackrel{\cong}{\lra} G_B\;\colon\; (h,n) \mapsto  (h, n + \tfrac{1}{2}B(h,h) - \tfrac{1}{2}C(h,h)),\]
is a group isomorphism.
\end{lemma}
\begin{proof}
With $S=B-C$:
\begin{align*}
 &\phi(h_1,n_1)\ast\phi(h_2,n_2) \\
 & = (h_1,n_1+\tfrac{1}{2}S(h_1,h_1)) \ast (h_2,n_2+\tfrac{1}{2}S(h_2,h_2)) \\
 & = (h_1+h_2,n_1+\tfrac{1}{2}S(h_1,h_1) + n_2+\tfrac{1}{2}S(h_2,h_2) + C(h_1,h_2)) \\
 & = (h_1+h_2,n_1+n_2+\tfrac{1}{2}S(h_1,h_1) + \tfrac{1}{2}S(h_2,h_2) + S(h_1,h_2) + B(h_1,h_2)) \\
 & = (h_1+h_2,n_1+n_2+\tfrac{1}{2}S(h_1+h_2,h_1+h_2) + B(h_1,h_2)) \\
 & = \phi(h_1+h_2,n_1+n_2+ B(h_1,h_2)) = \phi((h_1,n_1)\ast(h_2,n_2)). 
&
\end{align*}
\end{proof}

\begin{proof}[Proof of Proposition \ref{prop:gp1}]
Let $C = \tfrac{1}{2}(B-B^T)$ be the skew-symmetric part of $B$.
The exponential map for $G_B$ is the composite of the following three maps:
\[ \text{Lie}(G_B) \stackrel{\cong}{\lra} \text{Lie}(G_C) \stackrel{\exp}{\lra} G_C \stackrel{\phi}{\lra} G_B .\]
The first two of these maps are just the identity map $\RR^{p+q}\to \RR^{p+q}$
(by Lemmas \ref{gl2} and \ref{gl3}, respectively).
Lemma \ref{gl1} gives the explicit isomorphism $\phi\colon G_C\cong G_B$, 
with $C(h,h) = 0$.

\end{proof}

\bibliographystyle{amsxport}
\bibliography{MyBibfile}

\end{document}